\documentclass[11pt,bezier]{article}
\usepackage{amsmath,graphicx,amssymb,amsfonts}
\usepackage{color}

\newtheorem{prethm}{{\bf Theorem}}

\newenvironment{thm}{\begin{prethm}{\hspace{-0.5
               em}{\bf.}}}{\end{prethm}}

\newtheorem{pretheorem}{{\bf Theorem}}

\newenvironment{theorem}{\begin{pretheorem}{\hspace{-0.5
               em}{\bf.}}}{\end{pretheorem}}

\newtheorem{prepro}{Proposition}
\newenvironment{pro}{\begin{prepro}{\hspace{-0.5
               em}{\bf.}}}{\end{prepro}}

\newtheorem{prepr}{{\bf Theorem}}

\newenvironment{pr}{\begin{prepr}{\hspace{-0.5
               em}{\bf.}}}{\end{prepr}}

\newtheorem{predefinition}{Definition}
\newenvironment{definition}{\begin{predefinition}{\hspace{-0.5
               em}{\bf.}}}{\end{predefinition}}

\newtheorem{prelem}{Lemma}
\newenvironment{lem}{\begin{prelem}{\hspace{-0.5
               em}{\bf.}}}{\end{prelem}}

\newtheorem{precor}{Corollary}
\newenvironment{cor}{\begin{precor}{\hspace{-0.5
               em}{\bf.}}}{\end{precor}}

\newtheorem{preexam}{Example}
\newenvironment{exam}{\begin{preexam}{\hspace{-0.5
               em}{\bf.}}}{\end{preexam}}

\newtheorem{preremark}{Remark}
\newenvironment{remark}{\begin{preremark}{\hspace{-0.5
               em}{\bf.}}}{\end{preremark}}

\newtheorem{preexample}{{\bf Example}}

\newenvironment{example}{\begin{preexample}\em{\hspace{-0.5
               em}{\bf.}}}{\end{preexample}}

\newtheorem{preproof}{{\bf Proof.}}

\newenvironment{proof}[1]{\begin{preproof}{\rm
               #1}\hfill{$\Box$}}{\end{preproof}}

\newcommand{\rk}{{\rm rank}\,}

\title{\bf\ Some Results and Connections of an Eigendecomposition Problem}
\author{{\normalsize { M. Mohammad-Noori${}^{ \textrm{a}}$}, { N. Ghareghani${}^{ \textrm{b, c}}$}, {  M. Ghandi${}^{ \textrm{d}}$}\,
}\vspace{2mm} \\{\footnotesize{$^{ \textrm{a}}$\it
School of
Mathematics, Statistics and Computer
Science, University of Tehran, College of Science}}\vspace{-2mm}\\{\footnotesize{\it P.O. Box 14155-6455, Tehran,
Iran}}\\
{\footnotesize{$^{ \textrm{b}}$\it Department of Engineering Science, College of Engineering, University of Tehran,}}\vspace{-2mm}\\
{\footnotesize{\em P.O. Box 11165-4563, Tehran, Iran }}\\
{\footnotesize{$^{ \textrm{c}}$\it School of Mathematics, Institute for Research in
Fundamental Sciences {\rm(IPM),}}}\vspace{-2mm}\\{\footnotesize{\it P.O.Box: 19395-5746, Tehran,
Iran}}\\
{\footnotesize{$^{ \textrm{d}}$\it Broad Institute of MIT and Harvard
7 Cambridge Center, 4034C, Cambridge, MA 02142, United States of America}}\vspace{-2mm}\\
\\{\footnotesize Emails: morteza@ipm.ir, mmnoori@ut.ac.ir},\vspace{-2mm}\\
 {\footnotesize  ghareghani@ipm.ir, ghareghani@ut.ac.ir},\vspace{-2mm}\\
  {\footnotesize  mghandi@broadinstitute.org}
 }

\date{}
\begin{document}
\maketitle


\begin{abstract}
We consider the problem of finding nonzero eigenvalues and the corresponding eigenvectors of a matrix $AA^{\top}$, where $A$ is a special incidence matrix; This matrix can equivalently be defined based on a match relation between some sequences. By using a concrete description of the obtained eigenvectors, we show that these are pairwise orthogonal and satisfy nice properties. The combinatorial arguments, in the sequel, lead us to obtain formulas for entries of matrices $W$ and $WA$, where $W$ is the Moore-Penrose pseudo-inverse of $A$. A special case of this problem has previously found applications in computational biology.

\end{abstract}
\section{Introduction}
The two questions, by which, this work is motivated, has been appeared as parts of some works in computational biology and can roughly be described as follows: For a given incidence matrix $A$, which is defined in a ``special way", find explicit formulas for entries of $W$ and $WA$, where $W$ is the Moore-Penrose pseudo-inverse of $A$.
The definition of $A$ presented in \cite{kmer-b}, is based on a match relation between some sequences as follows:
Consider the integers $0\leq k \leq \ell$ and $b\geq 2$, a given $b$-letter alphabet $\Sigma_b$ (usually taken  $\Sigma_b=\{0,\cdots,b-1\}$), and an additional gap
(or blank) symbol ``$g$". Two elements $x, y \in \Delta_b:=\Sigma_b \cup \{g\}$ are {\it matchable} if either one of them is $g$ or $x=y$; Two sequences of length $\ell$ on $\Delta_b$ are {\it matchable} if they are matchable in all their positions, accordingly. Here, when we talk about matchable sequences, we are interested in the situation where one of the sequences is non-gapped, say $x\in \Sigma_b^{\ell}$ and $y\in \Delta_b^{\ell}$: In this situation $x$ and $y$ are matchable if they have the same value in the non-gapped positions of $y$; A {\it match position} (resp. {\it mismatch position}) is a value of $i$ with $1\leq i\leq \ell$, satisfying $y_i\neq g$ and $y_i=x_i$ (resp. $y_i\neq g$ and $y_i\neq x_i$), thus when we count the number of match (resp. mismatch) positions, we exclude the gapped positions.
The matrix $A$ (or $A_{\ell,k;b}$) is defined as a $(0,1)$-matrix whose columns (resp. rows) are indexed by sequences of length $\ell$ on $\Delta_b^{\ell}$ with no gap symbol (resp. with exactly $\ell-k$ gap symbols) and with $A(u,v)=1$ if and only if $u$ and $v$ are matchable. Note that by the above description, the column indexes are exactly the elements of $\Sigma_b^{\ell}$. In computational biology, we have $b=4$, $\Sigma_4=\{\texttt{A,C,G,T}\}$ and the set of row indexes and column indexes have special names: The set of column indexes, $\Sigma_4^{\ell}$ is non-gapped oligomers of length $\ell$, briefly called non-gapped $\ell$-mers. The set of row indexes is gapped oligomers with $k$ non-gapped positions and length $\ell$, briefly called $k$-mers (of length $\ell$).

 The matrices $W$ and $H=WA$ were applied to develop methods for analysis of biological sequences including DNA and protein sequences in \cite{kmer-b} and \cite{Enhanced kmer-b}. A prerequisite for computing entries of $W$ in the method used in \cite{kmer-b}, was finding an {\it orthonormal nonzero eigendecomposition} for $A A^{\top}$, that is finding matrices $Q$ and $\Lambda$ such that the diagonal matrix $\Lambda$ contains all nonzero eigenvalues of  $A A^{\top}$ and the matrix $Q$ consists of orthonormal eigenvectors of $AA^{\top}$ accordingly. In \cite{kmer-b} we applied matrix $W$ to find robust estimates for $\ell$-mer counts and showed that this significantly improves our ability to predict the binding of certain transcription factors to DNA sequences. In \cite{Enhanced kmer-b}, we used matrix $H$ to develop a method to efficiently compute the $\ell$-mer count estimates and to compute a string kernel based on them.

The matrix $A$ can also be interpreted as the incidence matrix of an incidence structure: If we identify each gapped sequence with the set of non-gapped sequences which are matchable with it, the row indexes (resp. column indexes) of $A$ correspond to blocks (resp. points).
 Interestingly, matrix $A$ has been previously appeared in \cite{Delsarte} based on the ``orthogonal array" concept. More precisely, Delsarte in \cite{Delsarte} defined a partial ordering $\preceq$ on $\Delta_b^{\ell}$ and referred $(\Delta_b^{\ell}, \preceq)$ as the Hamming semilattice. By that definition, the set of all
 sequences of length $\ell$ with $k$ non-blank positions can be considered as a fiber in that Hamming semilattice (see also \cite{Terwilliger}). He also showed that the Hamming semilattice is regular \cite{Delsarte association}.
In addition, he obtained $AA^{\top}$ as a linear combination of some idempotent matrices, which construct a basis for a special  $(n+1)$-dimensional algebra generated by characteristic matrices of distance $i$, $0\leq i\leq n$,
and computed the corresponding coefficients \cite{Delsarte association}.


Here we consider the two questions mentioned at the beginning of this introduction, but we replace the definition of $A$ with a more general form as follows: Let $(b_1,\cdots,b_{\ell})$ be a given sequence of integers with $b_i\geq 2$ ($1\leq i\leq \ell$). Recall that $\Sigma_{b_i}=\{0,\cdots,b_i-1\}$ and $\Delta_{b_i}=\Sigma_{b_i}\cup \{g\}$ for $i=1,\cdots,\ell$. Let $\Sigma_B=\Sigma_{b_1} \times \cdots \times \Sigma_{b_\ell}$ and $\Delta_B=\Delta_{b_1}\times  \cdots \times  \Delta_{b_\ell}$.
 The definition of matchable sequences is the same as the one mentioned earlier. Now we define the matrix $A$ (or $A_{\ell,k;B}$) as a $(0,1)$ matrix whose columns are indexed by the elements of $\Sigma_B$ and whose rows are indexed by the elements of $\Delta_B$ with exactly $\ell-k$ gap symbols and with $A_{\ell,k;B}(u,v)=1$ if and only if $u$ and $v$ are matchable. It is obvious that if $b_1=\cdots=b_{\ell}=b$, this definition is reduced to the older one, that is with $B=(b,\cdots,b)$, we have $A_{\ell,k;B}=A_{\ell,k;b}$. The new problem is finding entries of matrices $W$ and $H$, where $W$ is the Moore-Penrose pseudo-inverse of $A_{\ell,k;B}$ and $H=WA$.

It is worthwhile to discuss shortly about the biological application of the above generalization here:  This generalization allows us to have a mixture of features that are defined over alphabets of different lengths. For example, in addition to the DNA sequence that is defined over the alphabet \{\texttt{A,C,G,T}\}, one can also add DNA methylation status which is defined over \{\textit{methylated}, \textit{unmethylated}\} alphabet or any other discrete features. Then a similar methodology as used in \cite{kmer-b} and \cite{Enhanced kmer-b} can be applied to find a robust estimate of the joint distribution of the features using limited training data.

 Apart from the application mentioned above, the new problem leads us to extend and modify the combinatorial and linear algebraic tools used in \cite{kmer-b} to find answers of new questions. This sheds light on the previous study; Meanwhile, prepares (modified) generalized versions of some of the identities in the middle steps of the deduction. Section \ref{Ov-o-n} gives a basis for comparison between the old and new work; Particularly, some similarities and dissimilarities are discussed in the last paragraph of that section.

The overview of the old and new works in Sections \ref{Op} and \ref{On} is a prelude to a formal introduction of notations and preliminaries in Section \ref{notation}: General notations and definitions for sets, strings and sequences are presented in Section \ref{notationSet}; The definitions and some properties of the elementary symmetric polynomials are mentioned in Section \ref{sym sec}; Some preliminaries from linear algebra are discussed in Section \ref{PreLinAlg}. The function $\nu_B$ and some of its properties is defined and studied in Section \ref{nu sec}; The main results of this section, that is the identities given in Propositions \ref{mu.mu.sum}, \ref{someNuId} and \ref{NuIdentity}, are used in later sections. Using the definition of function $\nu_B$ and also the elementary symmetric polynomials, we propose an orthonormal basis for the eigenspaces of the matrix $AA^{\top}$ in Section \ref{eigen sec}. Finally, in Section \ref{w sec} we compute the entries of the matrix $W$, the pseudo inverse of $A$, and the matrix $H=WA$ using the results of the previous section. This section ends with an example containing the computation of orthonormal nonzero eigendecompsition of $AA^{\top}$ and matrices $W$ and $H$ for given parameters. We recommend having a first look on this example now and coming back to it during the study of the paper, whenever needed.

  \section{Overview of the problem}\label{Ov-o-n}
  The balk of the paper \cite{kmer-b} consists of the techniques used to find the entries of $W$ and since a similar approach is adapted in this manuscript, a brief overview of this part of the old work is given in Section \ref{Op}. Before studying this section, it is recommended to have a look at the errata of \cite{kmer-b} on \cite{website}. The results of Theorems \ref{EigenDcmp} and \ref{TheoremWEntries} are mentioned in \cite{kmer-b}, while the Theorem \ref{TheoremHEntries} is an exception, in the sense that it is motivated by \cite{Enhanced kmer-b} but its results are not formally mentioned in \cite{Enhanced kmer-b,kmer-b}: This theorem is mentioned among the old results because it can be easily deduced form them. A brief overview of the new work is given in Section \ref{On}; The correspondence between the three theorems mentioned there and the ones in Section \ref{Op} is easily observed. This section ends with a paragraph containing a comparison of the old and the new work from a close point of view.

   \subsection{Overview and summary of the previous work}\label{Op}
The first step of finding the entries of the desired matrices is to find an explicit orthonormal nonzero eigendecomposition for $AA^{\top}$ (Theorem \ref{EigenDcmp}). The next step is to give formulas for entries of $W$, the Moore-Penrose pseudo-inverse of $A$ (Theorem \ref{TheoremWEntries}). Finally, the last theorem (Theorem \ref{TheoremHEntries}) gives formulas for entries of $WA$.

To find an orthonormal nonzero eigendecomposition for  $A_{\ell k} A_{\ell k}^{\top}$, a linear order $\prec$ is imposed on  $\Delta_b:=\Sigma_b\cup\{g\}$ such that
 $$ 0\prec 1 \prec \cdots \prec b-1 \prec g$$
Then a function $\nu'$ is defined primarily on $\Delta_b\times \Delta_b$ by
    \[
        \nu'(x,y)=\left\{
            \begin{matrix}1 & \text{if }\,\, x \prec y\,\, \text{or} \,\, x=y=g,\\
                0 & \text{if } x\succ y,\\
                -y & \text{if } x=y\neq g.
            \end{matrix}
            \right.
    \]
and extended to the product set $\Delta_b^{\ell}\times \Delta_b^{\ell}$ by
\begin{equation} \label{mult} \nu'(x_1 \cdots x_{\ell},y_1 \cdots y_{\ell})=\nu'(x_1,y_1)\cdots \nu'(x_\ell,y_\ell),\end{equation}
where $x_i,y_i \in \Delta_b$, (See Remark \ref{nunu'}).
The useful identities satisfied by $\nu'$, helped us to find an orthonormal nonzero eigendecomposition for
 $AA^{\top}$ as described below.

 \begin{pr} \label{EigenDcmp} Let $A=A_{\ell, k;b}$. The matrix $AA^{\top}$ admits an orthonormal nonzero eigendecomposition $AA^{\top}=Q\Lambda Q^{\top}$.
  By using combinatorial arguments, the entries of $\Lambda$ are given in terms of hypergeometric functions and the entries of $Q$ are explicitly obtained
  as expressions containing $\nu'$ and hypergeometric functions.
 \end{pr}
 \begin{proof}
 {See Proposition 5 of \cite{kmer-b} and its proof.}
\end{proof}
\begin{remark}\label{nunu'} In \cite{kmer-b}, the name $\nu$ was used for the above function $\nu'$.
In this paper the name $\nu$ is used for the function on $\Delta_b^{\ell} \times \Delta_b^{\ell}$ satisfying the same multiplicative rule of $\nu'$ (Replace $\nu'$ by $\nu$ in (\ref{mult})) but with the following primary definition which holds for any $x,y \in \Delta_b$

	\[
		\nu(x,y)=\left\{
			\begin{matrix}-b & \text{if }\,\,  x=y=g,\\
1 & \text{if }\,\, x \prec y,\\				
0 & \text{if } x \succ y,\\
				-y & \text{if } x=y\neq g.
			\end{matrix}
			\right.
	\]
\end{remark}
This modification is justified in  Remark \ref{rem-nu-nu'}. The Theorem \ref{EigenDcmp} remains true if we replace $\nu'$ by $\nu$.

Considering Theorem \ref{EigenDcmp}, the Moore-Penrose pseudo-inverse of $A$ is obtained by $W=A^{\top} Q \Lambda^{-1} Q^{\top}$ (See Lemma \ref{W_nzed_A} in Section \ref{PreLinAlg}). To calculate the entries of $W$, firstly observe that by symmetry arguments,
 the entry $W_{\ell,k}(u,v)$ only depends on the number of mismatches between the $\ell$-mer $u$ and the gapped $k$-mer  $v$; Consequently, there exists a finite sequence of only $k+1$ values $t_0,t_1,\cdots,t_k$ such that
    $W_{\ell k}(u,v)=t_m$ if $u$ and $v$ have exactly $m$ mismatches. The Theorem \ref{TheoremWEntries} gives the exact values of the sequence $t_i$, for $0\leq i\leq k$.

     \begin{pr} \label{TheoremWEntries} Let $u\in \Sigma_b^{\ell}$, $v\in \Delta_b^{\ell}$ and $|v|_g=\ell-k$.
    The entries of matrix $W_{\ell, k;b}(u,v)$ is given by each of the following formulas, in which
     $m$ is the number of mismatches between $u$ and $v$:
        \begin{align}
W_{\ell k}(u,v)&=\frac{1}{b^\ell }\sum_{n=0}^k \sum_{t=0}^n \frac{1}{{\ell -n \choose \ell -k}} (-1)^{n-t}{k-m \choose t}{m \choose n-t}(b-1)^t \label{Wentries1}\\
W_{\ell k}(u,v)&=\frac{{k-\ell \choose m}}{b^{\ell}\binom{\ell}{k} \binom{k}{m}} \sum_{n=0}^{k-m}\binom{\ell}{n}(b-1)^{n}\label{Wentries2}
        \end{align}
    \end{pr}
 \begin{proof}
 {The second formula is mentioned in Proposition 7 of \cite{kmer-b}, while the first one is implicit in its proof.}
\end{proof}

    Applying symmetry arguments on the entries of $H_{\ell,k}$ shows that the entry $H_{\ell,k}(u,w)$ depends only on the number of mismatches between $u\in U_{\ell}$ and $v\in U_{\ell}$ and as in formula (10) of
    \cite{Enhanced kmer-b} we can simply use $H_{\ell k}=W_{\ell k} A_{\ell k}$ to obtain the formula $H_{\ell k}(u,w)=\sum_{i=0}^m {\ell-m \choose k-i}{m \choose i}t_i$, where $m$ is the number of mismatches between $u$ and $w$ and $t_i$ is as defined earlier.

    \begin{pr} \label{TheoremHEntries} The matrix $H:=WA$ is a symmetric idempotent matrix, the sum of entries of any of whose rows (columns) equals 1. Moreover, let $u,w\in \Sigma_b^{\ell}$ and let $u$ and $w$ differ in $\ell-p$ positions. Then the entries of $H=WA$ are given by the following formula:
\begin{equation}
H_{\ell k}(u,w)= \frac{1}{b^{\ell}}{\displaystyle\sum_{n=0}^{k}\,(-1)^{k-n} {\ell-p-1 \choose k-n}{p\choose n}(b-1)^n}\label{Hentries1}\\
\end{equation}

   \end{pr}
 \begin{proof}
 {This formula can be proved using equation (\ref{Wentries1}), the definition of $H$ and binomial coefficients. (One can use equation (\ref{Wentries2}) instead of (\ref{Wentries1}), but the proof is longer.)}
\end{proof}

 \subsection{Overview of this work}\label{On}
The main results of this paper are the three theorems listed in this section. These are generalizations of the theorems of the previous section. The following paragraph contains an informal representation of some notations and definitions required to express the main results (These definitions are formally presented in Section \ref{notation}).

 For a given sequence $B=(b_1,\cdots,b_{\ell})$ of integers and a given integer $1\leq i\leq \ell$, and elements $x,y \in \Delta_{b_i}$ the value $\nu_i(x,y)$ is defined as follows
	\[
		\nu_i(x,y)=\left\{
			\begin{matrix}-b_i & \text{if }\,\,  x=y=g,\\
1 & \text{if }\,\, x \prec y,\\				
0 & \text{if } x\succ y,\\
				-y & \text{if } x=y\neq g.
			\end{matrix}
			\right.
	\]
For each pair of words $v_1\cdots v_{\ell}$ and $w_1\cdots w_{\ell}$, with $v_i,w_i\in \Delta_{b_i}$, the value $\nu_B(v,w)$ is defined by the product rule $\nu_B(v,w)=\prod_{i=1}^{\ell}\nu_i(v_i,w_i)$.

\begin{theorem}\label{eigenDecomp} Let $B=(b_1,\cdots,b_{\ell})$ and $A=A_{\ell,k;B}$. The matrix $AA^{\top}$ admits an orthonormal nonzero eigendecomposition $AA^{\top}=Q\Lambda Q^{\top}$.
  Using combinatorial arguments, the eigenvalues are given in terms of elementary symmetric polynomials of  subsequences of $b_i$'s and the entries of $Q$ are explicitly obtained as expressions containing
  the function $\nu_B$ and elementary symmetric functions.
 \end{theorem}

From the above theorem, it is concluded that the Moore-Penrose pseudo-inverse of $A$ is $W:=A^{\top} Q \Lambda^{-1} Q^{\top}$. The following theorem gives the entries of this matrix.

\begin{theorem}\label{W-entries}
Let $u\in \Sigma_B$ and $v\in \Delta_B$ and $v$ contains exactly $\ell-k$ gapped positions. Let $G_v$ denotes the gapped positions of $v$ and $\overline{G}_v=\{1,\cdots,\ell\}\setminus G_v$. Let $P$ (resp. $Q$) be the set of nongapped positions of $v$ with $v_i=u_i$ (resp. with $v_i \neq u_i$). For a given set $G=\{a_1,\cdots,a_n\}\subseteq \{1,\cdots,\ell \}$ with $a_1<\cdots<a_n$, the sequence $B(G)$ is defined as $(b_{a_1},\cdots,b_{a_n})$. Then the entry $W_{\ell,k;B}(u,v)$ is given as below
\begin{equation*}
\label{w-ent-f}
W_{\ell k}(u,v)=\frac{1}{\displaystyle \prod_{i \in \overline{G}_v} b_i}\,\, \sum_{G}\frac{\displaystyle (-1)^{|Q\setminus G|}\prod_{i \in P\setminus G}(b_i-1)}{S_{\ell-k}(B(G))},
\end{equation*}
where the summation runs over all sets $G$ with $G_v \subseteq G \subseteq \{1,\cdots,\ell\}$.
\end{theorem}

\begin{theorem}\label{H-entries} The matrix $H_{\ell,k;B}=W_{\ell,k;B}A_{\ell,k;B}$ is a symmetric idempotent matrix, the sum of entries of any of whose rows (columns) equals 1. Let $u,w \in \Sigma_B$. Let $P$ (resp. $Q$) denote the set of positions $i$ in which $u_i=w_i$ (resp. $u_i\neq w_i$). Then the entry $H_{\ell,k;B}(u,w)$ is given as below
\begin{equation*}
\label{g-ent-f}
H_{\ell k}(u,w)=\frac{1}{\displaystyle \prod_{i=1}^{\ell} b_i}\,\, \sum_{G}{\displaystyle (-1)^{|Q\setminus G|}\prod_{i \in P\setminus G}(b_i-1)},
\end{equation*}
where the summation runs over all sets  $G \subseteq \{1,\cdots,\ell\}$ with $|G|\geq \ell-k$.
\end{theorem}

This paragraph is dedicated to a comparison between \cite{kmer-b} and the present work.
Usually the objects which have the same roles (in the old and new work) are given similar names (The matrices
$\Upsilon_{\ell k}$,  $C_{\ell,k}$, $H_{\ell k}$ in the new work are among the exceptions).
 There is an approximate correspondence between the sections as follows:
 The Sections 3.1, 3.2, 3.3, 3.4, 4 and 5 of \cite{kmer-b} correspond to Sections \ref{PreLinAlg}, \ref{sym sec}, \ref{notationSet}, \ref{nu sec}, \ref{eigen sec} and \ref{w sec} of the present work. Meanwhile, the reader should also note to differences, for instance the modification in the definition of the function $\nu$ causes several other slight modifications in the definitions and the results. This key helps the reader to focus on dissimilarities between Section 3.4 of \cite{kmer-b} and Section \ref{nu sec} of this work. Replacing the old recursive proof of orthogonality of the columns of $\Delta_{\ell,k}$ with a new direct proof of the same fact for $\Upsilon_{\ell, k}$ is also among the differences.

\section{Notations and Preliminaries}\label{notation}
This section contains three subsections: In \ref{notationSet} we fix some notations for sets, strings and sequences  that we use in this work. In \ref{sym sec} we study some properties of elementary symmetric polynomials, and in \ref{PreLinAlg} we give some preliminaries and notations from linear algebra.

\subsection{Notations for sets strings and sequences}\label{notationSet}
\begin{definition} Let $\ell$ be a positive integer. The set $[\ell]$ is defined as $[\ell]=\{1,\cdots,\ell\}$. For a set $X$ and a nonnegative integer $n$, by ${X \choose n}$, we mean the set of all $n$-element subsets of $X$. Thus $|{X \choose n}|={|X| \choose n}$ and $|{[\ell] \choose n}|={\ell \choose n}$.
\end{definition}

\begin{definition}
 A word $x$ on the alphabet $\Sigma$, is a sequence $x=x_1\cdots x_{\ell}$ whose elements $x_i$ belong to the finite set $\Sigma$.
As in \cite{kmer-b} that for a given integer $b\geq 2$, the sets $\Sigma_b$, $\Delta_b$, $\Gamma_b$ are defined as follows
	\begin{align*}
        \Sigma_b &= \{0,1,\cdots,b-1\}\\
		\Delta_b &= \Sigma_b \cup \{g\}\\
		\Gamma_b &= \Delta_b \setminus \{0\},
	\end{align*}
where $g$ stands for the gap symbol.
\end{definition}

\begin{definition}
Let $B=(b_1,b_2,\cdots,b_{\ell})$ be an $\ell$-tuple of integers $b_i\geq 2$. Define the sets $\Sigma_B$, $\Delta_B$, $\Gamma_B$, $U_{\ell; B}$ and $V_{\ell,k;B}$ as follows
\begin{align*}
        \Sigma_B &= \Sigma_{b_1}\times \cdots \times \Sigma_{b_{\ell}}\\
        \Delta_B &= \Delta_{b_1}\times \cdots \times \Delta_{b_{\ell}}\\
        \Gamma_B &= \Gamma_{b_1}\times \cdots \times \Gamma_{b_{\ell}}\\
        U_{\ell;B} &= \Sigma_B\\
        V_{\ell,k;B} &= \{v\in \Delta_B : |v|_g=\ell-k\}\\
        V'_{\ell,k;B}&=\{w\in \Gamma_b^\ell: |w|_g=\ell-k\}\\
        V_{\ell, \leq k;B}&=\bigcup_{m=0}^k V_{\ell m}\\
        V'_{\ell, \leq k;B}&=\bigcup_{m=0}^k V'_{\ell m}
\end{align*}
\end{definition}
In the following definition we study some special orderings defined on the sets we introduced before. As usual, if $\prec$ is a given strong order on a set, the corresponding weak order is defined as $\preceq\, = \,(\prec \cup =)$ and if
$\preceq$ is given as a weak order on a set, the corresponding strong order, $\prec$ is defined by $\prec \, =\, (\preceq \cap \neq)$.

\begin{definition}
For a given integer $b\geq 2$, we order the set $\Delta_b$ by the following linear order
$$0\prec 1\prec \cdots \prec b-1 \prec g\, .$$
For any positive integer $\ell$, this order induces a partial order $ \prec'$ on $\Delta_b^{\ell}$ ;More precisely, the corresponding weak order, $\preceq'$, is given as below
$$ \preceq' \, = \, \preceq \times \cdots \times \preceq\, .$$
on the product set $\Delta_b^{\ell}$. Here we slightly extend this definition to a more general situation and define a partial order on the set $\Delta_B$. For any $1\leq i\leq \ell$ we consider the order $\prec_i$ on the set $\Delta_{b_i}$ defined by
$$0\prec_i 1\prec_i \cdots \prec_i b_i-1 \prec_i g\, $$
and we let
\begin{equation}
\preceq_B\, =\, \preceq_1\times \cdots \times \preceq_\ell
\end{equation}
It is clear that the relation $\prec_B$ is a partial order; Moreover, if $\ell>1$, then it is not a total order.
\end{definition}

\begin{exam}
Let $B=(2,3)$, $\ell=2$ and $k=1$. Then we have
\begin{align*}
        \Delta_B &= \{00,01,02,0g,10,11,12,1g,g0,g1,g2,gg\}\\
        \Gamma_B &= \{11,12,1g,g1,g2,gg\}\\
        U_{2;B} &= \{00,01,02,10,11,12\}\\
        V_{2,1;B} &= \{0g,1g,g0,g1,g2\}\\
        V'_{2,1;B}&=\{1g,g1,g2\}\\
        V_{2, \leq 1;B}&=\{0g,1g,g0,g1,g2,gg\}\\
        V'_{2, \leq 1;B}&=\{1g,g1,g2,gg\}
\end{align*}

\end{exam}

\begin{remark} As it is clear from the definitions of $U_{\ell;B}$ and  $V_{\ell,k;B}$, when we use these notations we specially emphasize on parameters $\ell$ and $k$.
\end{remark}

 \begin{definition} For any word $v \in \Delta_B$ we set $G_v=\{i: 1\leq i \leq \ell, v_i=g\}$ and $\overline{G}_v = \{1,\cdots,\ell\}\setminus G_v$. If $X=\{x_1,\cdots,x_n\}$ is a subset of $\{1,\cdots,\ell\}$ with $x_1<x_2<\cdots<x_n$ then by $B(X)$ we mean $(b_{x_1},\cdots,b_{x_n})$. Especially, if $v\in \Delta_B$ and  $v'\in \Gamma_B$, then
 $B(G_{v})=(b_i)_{i \in G_{v}}$ and $B(G_{v'})=(b_i)_{i \in G_{v'}}$.
 \end{definition}

 \begin{exam}\label{Bgv} Let $B=(b_1,\cdots,b_7)$, $b_1=2, b_2=b_3=b_4=3$ and $b_5=b_6=b_7=4$. Let $v'=1g2gg12$. Then $G_{v'}=\{2,4,5\}$, $B(G_{v'})=(b_2,b_4,b_5)=(2,3,4)$
\end{exam}

 \begin{definition}
    We say elements $u\in \Sigma_B$ and $v\in \Delta_{B}$ match (or $u$ and $v$ are matchable) if for any $1\leq i\leq \ell$ with $v_i\neq g$ we have $u_i=v_i$; We denote this by $v\sim u$. The set of the elements $v\in V_{\ell, k;B}$ which are matchable with $u\in U_{\ell;B}$, is denoted by $M_{\ell, k}(u)$.
    The set of elements $u\in U_{\ell;B}$ which are matchable with $v$, is denoted by $M'_{\ell, k;B}(v)$.
\end{definition}

\begin{definition} The matrix $A_{\ell, k;B}$ is defined as a $(0,1)$ matrix whose rows and columns indexed respectively by the elements of $\Delta_B$ and $\Sigma_B$ and $A_{\ell, k;B}(w,v)=1$ if and only if $w$ is matchable with $v$ in non-gapped positions.
\end{definition}

\subsection{Elementary symmetric polynomials and some identities}\label{sym sec}

Elementary symmetric polynomials are well-studied objects in the study of polynomials ring $k[x_1,x_2,\cdots,x_n]$ (see Chapter $7$ of \cite{cox}). The fundamental theorem of symmetric polynomials state that every symmetric
polynomial in  $k[x_1,x_2,\cdots,x_n]$ can be written uniquely as a polynomial in terms
elementary symmetric functions (A well-known example of this fact from the college algebra, is the Waring formula stating the expression $x_1^k+x_2^k$ in terms of $s=x_1+x_2$ and $p=x_1 x_2$). In this section after a formal definition we introduce some of the identities satisfied by them.

\begin{definition} Let $i$ and $n$ be nonnegative integers and let $X=(x_1,x_2,\cdots,x_n)$ be a finite sequence of  variables. The $i$-th elementary symmetric polynomial, denoted as $S_i(X)$, is defined as $S_i(X):=\sum_{I\in {X \choose i}} \prod_{i\in I} x_i$.
\end{definition}

 \begin{exam}\label{SBgv} Considering Example \ref{Bgv}, we have
\begin{align*}
S_0(B(G_{v'}))&=1\\
S_1(B(G_{v'}))&=b_2+b_4+b_5=9\\
S_2(B(G_{v'}))&=b_2 b_4+b_2 b_5+b_4 b_5=26\\
S_3(B(G_{v'}))&=b_2 b_4 b_5=24\\
S_4(B(G_{v'}))&=0
\end{align*}
\end{exam}

 Note that if $x_i=x$, \,$i=1,\cdots,n$, then $S_i(X)={n \choose i}x^i$; thus the mentioned identities can be considered as generalizations of binomial identities. Particularly, part (i) of the following lemma gives a generalization of Pascal's identity, where part (ii) generalizes the binomial theorem.

\noindent {\bf{Notation.}} Let $X=(x_1,\ldots,x_n)$ be a finite sequence of numbers and $\alpha$ and $\beta$ be arbitrary numbers. Then we show the sequence $(\beta x_1+\alpha,\ldots,\beta x_n+\alpha)$ by $\beta X+\alpha$; Particularly $X+\alpha=(x_1+\alpha,\cdots,x_n+\alpha)$ and $\beta X=(\beta x_1,\cdots,\beta x_n)$. For a given integer $1\leq m\leq n$, by $X_{\widehat{m}}$ we mean the sequence obtained by deleting the $m$-th position of $X$, that is
$X_{\widehat{m}}=(x_1,\cdots,x_{m-1},x_{m+1},\cdots,x_n)$.
The following lemma is easy to prove.

 \begin{lem} \label{S_iIdnt} Let $X=(x_1,\ldots,x_n)$ be a finite sequence of variables.
    \begin{itemize}
        \item[\rm (i)] We have $S_0(X)=1$ and $S_{n}(X)=\prod_{i=1}^n x_i$. Moreover, for any integer $i$
 and any integer $1\leq m\leq n$ we have
 \begin{equation}\label{Pascal-s}
 S_i(X)=x_m S_{i-1}(X_{\widehat{m}})+S_i(X_{\widehat{m}})
  \end{equation}
        \item[\rm (ii)] We have
\begin{align}\label{bino}
\sum_{i=0}^{n} S_i(X)\beta^i \alpha^{k-i}&=\prod_{i=1}^n (\beta x_i+\alpha)\\
\sum_{i=0}^n S_i(X)&=\prod_{i=0}^n (x_i+1)\\
\end{align}
Consequently, for any subset $M\subseteq [n]$ we have
\begin{align}
\sum_{I\subseteq M}\prod_{i\in I}x_i \beta^{|I|}\alpha^{|M|-|I|}&=\prod_{i\in M}(\beta x_i+\alpha)\\
\sum_{I\subseteq M}\prod_{i\in I}x_i&=\prod_{i\in M}(x_i+1)\label{prodx+1}
\end{align}
    \end{itemize}
 \end{lem}

\begin{definition} Let $i$ and $n$ be nonnegative integers and let $X=(x_1,x_2,\cdots,x_n)$ be a finite sequence of  variables. The expression $R_i(X)$ is then defined as follows:
\begin{equation} \label{Ri-I}
R_i(X)=\sum_{j=0}^i S_j(X-1)
\end{equation}
\end{definition}
 The following lemma which gives a binomial type recurrence relation for the entries $R_i$, is easily proved by the definition of $R_i$ and Lemma \ref{S_iIdnt} (i) .

 \begin{lem} \label{R_iIdnt} Let $X=(x_1,\ldots,x_n)$ be a finite sequence of variables.
   We have $R_0(X)=1$ and for any integer $i\geq n$, $R_i(X)=\prod_{i=1}^n x_i$. Moreover, for any integer $i$
 and any $1\leq m\leq n$ we have
 \begin{equation}
 R_i(X)=(x_m-1)R_{i-1}(X_{\widehat{m}})+R_i(X_{\widehat{m}})
  \end{equation}
  \end{lem}

\begin{exam}
Let $0\leq k\leq \ell$ be integers and $B=(b_1,\cdots,b_{\ell})$, $u\in U_{\ell;B}$ and $v\in V_{\ell k;B}$.
Then
\begin{align*}
        &|U_{\ell;B}|= \prod_{i=1}^{\ell} b_i, \,\,\,  &|V_{\ell,k;B}|=S_k(B), \,\,\,          &|V'_{\ell,k;B}|=S_k(B-1), \,\,\,&|V_{\ell,\leq k;B}|= R_k(B+1)\\
        &|V'_{\ell,\leq k;B}|=R_k(B), \,\,\,  &|M_{\ell,k}(u)|={\ell \choose k}, \,\,\,
        &|M'_{\ell,k}(v)|=\prod_{i\in G_v} b_i  \,&
\end{align*}
Moreover, If we denote the hamming distance of two elements $u,w\in U_{\ell;B}$ by $Ham(u,w)$, then
we have
\begin{align*}
        |\{w\in U_{\ell;B}: Ham(w,u)=d| &= S_d(B-1)\\
        |\{w\in U_{\ell;B}: Ham(w,u)\leq d| &= R_d(B)
\end{align*}

\end{exam}
The following lemma contains an identity which is useful for future applications.
\begin{lem}\label{s_i x 2}
Let $B=(b_1,\ldots,b_{\ell})$. Then the following identity holds:
\begin{equation}\label{s_i x 2 eq}
\sum_{\overline{G}\subseteq [\ell], \,\,\,|\overline{G}|\leq k} \prod_{i\in \overline{G}}(b_i-1)\, S_{\ell-k}(B(G))
={{\ell}\choose{k}}\prod_{i=1}^{\ell}b_i
\end{equation}
\end{lem}

\begin{proof}
{
\begin{align*}
\sum_{\overline{G}\subseteq [\ell], \,\,\,|\overline{G}|\leq k}S_{\ell-k}(B(G)) \,\,\prod_{i\in \overline{G}}(b_i-1)
&=\sum_{\overline{G}\subseteq [\ell], \,\,\,|\overline{G}|\leq k}\,\,\,\left( \sum_{N\subseteq G,\,\, |N|=\ell-k}\,\, \prod_{i\in N} b_i \right)\,\,\prod_{i\in \overline{G}}(b_i-1)\\
&=\sum_{\overline{G}\subseteq [\ell], \,\,\,|\overline{G}|\leq k}\,\,\,\sum_{\overline{N}\supseteq \overline{G},\,\, |\overline{N}|=k}\,\, \prod_{i\in N} b_i \,\,\prod_{i\in \overline{G}}(b_i-1)\\
&=\sum_{\overline{N}\subseteq [\ell], \,\,|\overline{N}|=k}\,\,\,\,\sum_{\overline{G}\subseteq \overline{N}}\prod_{i\in N} b_i\,\,\prod_{i\in \overline{G}}(b_i-1)\,\, \\
&=\sum_{\overline{N}\subseteq [\ell], \,\,|\overline{N}|=k}\,\,\,\prod_{i\in N} b_i \,\sum_{\overline{G}\subseteq \overline{N}}\,\,\prod_{i\in \overline{G}}(b_i-1)\\
&=\sum_{\overline{N}\subseteq [\ell], \,\,|\overline{N}|=k}\,\,\,\prod_{i\in N} b_i \,\prod_{i\in \overline{N}} b_i~~~\hbox{(by (\ref{prodx+1}))}~~~ \\
&=\sum_{\overline{N}\subseteq [\ell], \,\,|\overline{N}|=k}\,\,\,\prod_{i=1}^{\ell} b_i \\
&=\prod_{i=1}^{\ell} b_i \sum_{\overline{N}\subseteq [\ell], \,\,|\overline{N}|=k}1 \\
&={{\ell}\choose{k}}\prod_{i=1}^{\ell}b_i
\end{align*}
}
\end{proof}

\subsection{Notations and preliminaries from Linear Algebra}\label{PreLinAlg}
    Let $F$ be a field and $A\in F^{m\times n}$ be a matrix. The row space of $A$ is denoted as ${\rm row}(A)$, the column space of $A$ is denoted as ${\rm col}(A)$, and the dimension of the row space (which is the same as the dimension of the column space)
    of $A$ is denoted as ${\rm rank}(A)$.
    The kernel of $A$, denoted as ${\ker}(A)$, is the space of all column vectors $x$ satisfying $Ax={\bf 0}$ and the dimension of this space is called the nullity of $A$ and denoted as ${\rm null} (A)$.
    It is known that ${\rm null}(A)+{\rm rank}(A)=n$.
    Let $B\in F^{n\times n}$. The characteristic polynomial of $B$ is defined as $p_B(z)=\det (zI-X)$.
    An element $\lambda \in F$ is an eigenvalue of $B$ if there exists a nonzero column vector $x$ satisfying $Bx=\lambda x$; The vector $x$ is called an eigenvector of $B$.
    It is observed that $\lambda$ is an eigenvalue of $B$ if and only if it is a root of the characteristic equation $p_B(z)$.
    For an eigenvalue $\lambda$, the space $\ker(B-\lambda I)$ is called the eigenspace of $B$ corresponding to $\lambda$.
    The {\it algebraic multiplicity} of an eigenvalue $\lambda$, denoted as $\alpha(\lambda)$ is the multiplicity of the root $\lambda$ of $p_B(z)$.
    The geometric multiplicity of an eigenvalue $\lambda$, denoted by $\gamma(\lambda)$, is the dimension of ${\ker}(B-\lambda I)$. It is well known that for any matrix $B$ and any eigenvalue $\lambda$ of $B$ we have
    $\gamma(\lambda)\leq \alpha(\lambda)$.
    The matrix $B$ is called diagonalizable if there exists a nonsingular matrix $P$
     such that $B=P\Lambda_0 P^{-1}$ for some diagonal matrix $\Lambda_0$.
     It is easily seen that for a diagonalizable matrix $B$ we have $\gamma(\lambda)\leq \alpha(\lambda)$. Indeed, all eigenvalues of $B$ appear in the main diagonal of  $\Lambda_0$ and the columns of $P$ are the corresponding  eigenvectors.


    If eigenvectors belonging to distinct eigenvalues of the matrix $B$ are mutually orthogonal, then there exists an
     eigendecomposition $B=P\Lambda_0 P^{-1}$ such that the columns of $P$ are mutually orthogonal normal (orthonormal) vectors. It is easy to see that if the columns of $P$ are orthonormal, then $P^{-1}=P^{\top}$.
       Hence, if the columns of $P$ are orthonormal, then $B=P \Lambda_0 P^{\top}$. We may assume that the nonzero eigenvalues appears before the zeros on the main diagonal of $\Lambda_0$; Consequently, we obtain the block decomposition $P=[Q\,N]$ where the columns of $N$ are in ${\ker} (B)$. This gives the eigendecomposition $B=Q \Lambda Q^{\top}$, where the matrix $Q$  is obtained by deleting the columns of $P$ which are in ${\ker} (B)$, and
 $\Lambda$ is obtained by deleting the zero columns and zero rows of $\Lambda_0$, in this paper such a matrix decomposition is called {\it orthonormal nonzero eigendecomposition}.

    For a matrix $A\in \mathbb {C}^{m\times n}$, its Hermitian adjoint, $A^*$, is its conjugate transpose, i.e. $A^*$ is an $n \times m$ matrix with $A^*(i,j)=\overline{A(j,i)}$.
    A matrix $A$ is {\it Hermitian} if $A=A^*$, thus a real matrix $A$ is Hermitian if and only if it is symmetric.
     It is well known that all eigenvalues of a Hermitian matrix $A$ with dimension $n$ are real, and that $A$ has $n$ linearly independent eigenvectors, that is for every eigenvalue $\lambda$ of $A$ we have $\alpha(\lambda) = \gamma(\lambda)$. Moreover, Hermitian matrix has orthogonal eigenvectors for distinct eigenvalues. Hence, every
     Hermitian matrix $A$ has an orthonormal nonzero eigendecomposition of the form $A=Q \Lambda Q^{\top}$.

     A Hermitian matrix $A$ is {\it positive definite} (resp. {\it positive semidefinite}) if ${\rm Rel}(x^*Ax) > 0$ (resp. ${\rm Rel}(x^*Ax) \geq 0$) for all nonzero $x\in \mathbb{C}^n$. It is concluded that
       areal symmetric matrix $A$ of order $n$ is positive definite (resp. positive semi-definite) if $x^{\top} Ax>0$ (resp. $x^{\top} Ax\geq 0$) for all nonzero $x\in \mathbb{R}^n$. For any matrix $A$, the matrix $A^{\top}A$ is positive semidefinite, and ${\rm rank}(A) = {\rm rank}(A A^{\top})$. Conversely, any positive semidefinite matrix $M$ can be written as $M = A^{\top}A$; this is the Cholesky decomposition. A Hermitian (or symmetric) matrix is positive definite (resp. positive semi-definite) if and only if all its eigenvalues are positive (resp. nonnegative). The following lemma is concluded from the facts that mentioned above.

       \begin{lem}\label{positivesemidef}
       Let $B$ be a positive semi-definite matrix. Then $B$ admits an orthonormal nonzero eigendecomposition of the form $B=Q \Lambda Q^{\top}$. Where
       \begin{equation}\label{qqtI}
       Q^{\top}Q=I.
       \end{equation}
       Moreover let $B=AA^{\top}$, then we have
       \begin{equation}\label{AtQQT}
       A^{\top}QQ^{\top}=A^{\top}.
       \end{equation}
       \end{lem}
       \begin{proof}{
        Using the previous notation, let $AA^{\top}=P\Lambda_0 P^{\top}$ be an orthonormal decomposition for $AA^{\top}$ and $P=[Q\, N]$ where the columns of $N$ are in ${\ker} (AA^{\top})$. The equation (\ref{qqtI}) is concluded from the orthonormality of the columns of $Q$. To prove (\ref{AtQQT}), we claim that
        \begin{equation}\label{ATN0}
       A^{\top}N=0.
       \end{equation}
       To prove this, note that if $y\in {\ker}(A_{\ell k}A_{\ell k}^{\top})$, then from $A_{\ell k}A_{\ell k}^{\top}y=0$ we obtain $y^{\top}A_{\ell k}A_{\ell k}^{\top}y=0$ and $||A_{\ell k}^{\top}y||=0$, thus $A_{\ell k}^{\top}y=0$, which shows (\ref{ATN0}). Now, from $PP^{\top}=I$ we obtain $QQ^{\top}+NN^{\top}=I$; Multiplying from left by $A^{\top}$ and using (\ref{ATN0}), we provide (\ref{AtQQT}), as required.
       }
       \end{proof}



    The {\it Moore-Penrose pseudoinverse} of a matrix $A$, denoted by $A^{+}$, is defined as a matrix that satisfies all the following four conditions:
     \begin{align*}
        AA^{+}A&=A \\
        A^{+}AA^{+}&=A^{+} \\
        (AA^{+})^{^*}&=AA^{+} \\
        (A^{+}A)^{^*}&=A^{+}A.
    \end{align*}

    The Moore-Penrose pseudoinverse exists and is unique for any given matrix $A$. Also we have $ A^{+}= (A^* A)^{+} A^*= A^*(AA^* )^{+}$. For further properties of the Moore-Penrose pseudoiverse see for instance \cite{handbooklin}. The following lemma gives the Moore-Penrose pseudoinverse of a matrix $A$ using an orthonormal nonzero eigendecomposition of $AA^{\top}$.

\begin{lem}\label{W_nzed_A}
Let $AA^{\top}=Q \Lambda Q^{\top}$ be an orthonormal nonzero eigendecomposition of $AA^{\top}$ and let $W=A^{\top}Q\Lambda^{-1}Q^{\top}$. Then $W$ is the Moore-Penrose pseudoinverse of $A$.
\end{lem}

\begin{proof}
{Firstly, we have
\begin{align*}
    WAW&=(A^{\top}Q\Lambda^{-1}Q^{\top})A (A^{\top}Q\Lambda^{-1}Q^{\top})\\
    &=A^{\top}Q\Lambda^{-1}Q^{\top}(A A^{\top})Q\Lambda^{-1}Q^{\top}\\
    &=A^{\top}Q\Lambda^{-1}Q^{\top}(Q\Lambda Q^{\top})Q\Lambda^{-1}Q^{\top}\\
    &=A^{\top}Q(\Lambda^{-1}Q^{\top} Q\Lambda) (Q^{\top} Q)\Lambda^{-1}Q^{\top}\\
    &=A^{\top}Q \Lambda^{-1}Q^{\top}~~\hbox{(by (\ref{qqtI}))}~~\\
    &=W.
    \end{align*}

    Secondly, we have

    \begin{align*}
    AWA&=A(A^{\top}Q\Lambda^{-1}Q^{\top})A\\
    &=(AA^{\top})Q\Lambda^{-1}Q^{\top}A\\
    &=Q(\Lambda Q^{\top} Q\Lambda^{-1})Q^{\top}A\\
    &=Q Q^{\top}A~~\hbox{(by (\ref{qqtI}))}~~\\
    &=A ~~\hbox{(by (\ref{AtQQT})).}~~
    \end{align*}
     Finally, observe that $WA$ and $AW$ are real symmetric matrices.
     It is concluded that $W$ is the Moore-Penrose pseudo-inverse of $A$.
}
\end{proof}
\section{The function $\nu_B$ and some of its properties}\label{nu sec}

In this section, we consider an order on the set $\Delta_B$ and based on this define a function $\nu_B$ on the set $\Delta_B \times \Delta_B$ and inspect some of its properties; Some of the identities help us
 to find mutually orthogonal eigenvectors (corresponding to non-zero eigenvalues) and others are useful in future computations of the entries of some matrices. In fact, the component of the eigenvectors of the matrix $AA^{\top}$ can be expressed in terms of $\nu_B$ and using the properties of $\nu_B$ we will prove that these eigenvectors are mutually orthogonal.

For a given integer $b\geq 2$, the following linear order makes $\Delta_b$ a totally ordered set:
$$0\prec 1\prec \cdots \prec b-1\prec g$$
and when this order is induced on the product set $\Delta_b^{\ell}$ for an integer $\ell\geq 1$, a poset is obtained; More precisely, for two elements $x=x_1\cdots x_{\ell}$  and $y=y_1\cdots y_{\ell}$ with $x_i,y_i\in \Delta_b$, $(1\leq i \leq \ell)$, we have $x\preceq y$ if and only if $x_i \preceq y_i$ holds for $i=1,\cdots,\ell$. Below is presented the definition of two functions on $\Delta_b^{\ell}\times \Delta_b^{\ell}$.

 \begin{definition}\label{nuPrim-def} The function $\nu'$ is primarily defined on $\Delta_b\times \Delta_b$ by
\[
		\nu'(x,y)=\left\{
			\begin{matrix}1 & \text{if }\,\,  x=y= g,\\
                    -y & \text{if }\,\,  x=y\neq g,\\
                    1 & \text{if }\,\,  x\prec y,\\
				0 & \text{if } x\succ y.
			\end{matrix}
			\right.
	\]
and extended then to $\Delta_b^{\ell}\times \Delta_b^{\ell}$ by the product rule
\begin{equation*}
\nu'(x_1\cdots x_{\ell},y_1\cdots y_{\ell})=\nu'(x_1,y_1)\cdots\nu'(x_{\ell},y_{\ell})
\end{equation*}
\end{definition}

 \begin{definition}\label{nu-def} The function $\nu$ is primarily defined on $\Delta_b\times \Delta_b$ by
\[
		\nu(x,y)=\left\{
			\begin{matrix}-b & \text{if }\,\,  x=y=g,\\
-y & \text{if } x=y\neq g,\\
1 & \text{if }\,\, x\prec y,\\				
0 & \text{if } x\succ y.
			\end{matrix}
			\right.
	\]
and extended then to $\Delta_b^{\ell}\times \Delta_b^{\ell}$ by the product rule
\begin{equation*}
\nu(x_1\cdots x_{\ell},y_1\cdots y_{\ell})=\nu(x_1,y_1)\cdots\nu(x_{\ell},y_{\ell})
\end{equation*}
\end{definition}
\begin{remark} \label{rem-nu-nu'} The function $\nu'$ was firstly defined in \cite{kmer-b} (under the name $\nu$) to give a concrete description for an orthogonal non-zero eigendecomposition of $AA^{\top}$. It was also possible to use there, the function $\nu$ presented in the Definition \ref{nu-def}, for the same purpose. Moreover, as seen in the rest of this work, the advantage of $\nu$ is that it can be used in the current more general problem, as well.
\end{remark}
\begin{remark}\label{incidenceAlg}
The function $\nu'$ satisfies the property `` $\nu'(x,y)=0$ unless $x\preceq y$" and so does $\nu$; In other words, these functions are elements of the incidence algebra of the poset $\Delta_b^{\ell}$ (For the definition and some examples of this concept, see for instance Chapter 8 of  \cite{CameronNotes}).
\end{remark}

In the generalized problem, instead of one value of $b$ we have an $\ell$-tuple $(b_1,\cdots,b_{\ell})$, so it is natural to define $\ell$ functions $\nu_1,\cdots,\nu_{\ell}$ and obtain a new function $\nu_B$ defined on the product set $\Delta_B \times \Delta_B$ from their product; The following definition makes this procedure clear.

\begin{definition}\label{nui-nuB-def}
Consider the $\ell$-tuple $B=(b_1,b_2,\cdots,b_{\ell})$, where $b_i\geq 2$ is integer for $i=1,\cdots,\ell$.
For any $i$,\, \, $(1\leq i\leq \ell)$, we define the function $\nu_{i}$ on $\Delta_{b_i} \times \Delta_{b_i}$ as
	\[
		\nu_i(x,y)=\left\{
			\begin{matrix}-b_i & \text{if }\,\,  x=y=g,\\
-y & \text{if } x=y\neq g,\\
1 & \text{if }\,\, x\prec y,\\				
0 & \text{if } x\succ y.
			\end{matrix}
			\right.
	\]
Now the function $\nu_B$ is defined on the product set $\Delta_B\times \Delta_B$ by the following product rule
\begin{equation}\label{nuB}
\nu_B(x_1\cdots x_{\ell},y_1\cdots y_{\ell})=\prod_{i=1}^{\ell}\nu_i(x_i,y_i)
\end{equation}
\end{definition}

\begin{remark}
It is observed that this function is an element of the incidence algebra of the poset $\Delta_{B}$ which satisfies
	\[
		\sum_{x \preceq  z \preceq  y} \nu_B(x,z)=\left\{
			\begin{matrix}\prod_{i=1}^{\ell}(y'_i-2x'_i) & \text{if }\,\,  x\preceq y,\\
				0 & \text{otherwise. }
			\end{matrix}
			\right.
	\]
where the values $x'_i$, $(1\leq i\leq \ell),$ are defined
	\[
		x'_i=\left\{
			\begin{matrix} b_i & \text{if }\,\,  x_i=g,\\
x_i & \text{otherwise. }
			\end{matrix}
			\right.
	\]
and $y'_i$'s are defined similarly.
\end{remark}


 Some useful identities about $\nu_B$ are stated in Proposition \ref{mu.mu.sum}, but before presenting that, we need some definitions and lemmas.
 \begin{definition}\label{a0a1a2a3}  Let $\ell, k$ be integers with $0\leq k\leq \ell$ and let $v', v'' \in V_{\ell, \leq k} (B)$.
 Let $m,n$ be integers with $0\leq m,n \leq \ell$ such that $|G_{v'}|= \ell -n$ and $|G_{v''}|= \ell -m$. Define the sets $A_3, A_2, A_1$ and $A_0$ by
 $A_3=\overline{G}_{v'}\cap \overline{G}_{v''}$, $A_2=G_{v''}\setminus G_{v'}$, $A_1=G_{v'}\setminus G_{v''}$
 and $A_0= G_{v'}\cap G_{v''}$.
 \end{definition}

 \begin{lem} \label{A_is} Let $v',v'' \in V_{\ell,\leq k}$ and the sets $A_0$, $A_1$,$A_2$ and $A_3$ be as in Definition \ref{a0a1a2a3}.
    \begin{itemize}
        \item[\rm (i)] The sets $A_3, A_2, A_1$ and $A_0$ are mutually disjoint and
            $$A_0\cup A_1 \cup A_2 \cup A_3=\{1,2,\ldots \ell\}$$
              Moreover $A_0\neq \{1,2,\ldots \ell\}$ unless $v'=v''=g^{\ell}$.
        \item[\rm (ii)] If $A_1=A_2=\emptyset$, then $A_0=G_v=G_{v'}$ and $\overline{G}_{v'}=\overline{G}_{v''}=A_3$; If furthermore $v'\neq v''$, then there exists $i\in A_3$ such that $v'_i\neq v''_i$
    \end{itemize}
 \end{lem}
\begin{proof}{The proof is straightforward. }
\end{proof}
\begin{lem}\label{mu.mu} Let $w\in V_{\ell k}(B)$ and $v',v'' \in \Gamma_B$.
  \begin{itemize}
        \item[\rm (i)] If $\nu_B (w,v')\nu_B (w,v'')\neq 0$, then $G_w \subseteq A_0$.
        \item[\rm (ii)] If $G_w \subseteq A_0$, then $\nu_B (w,v')\nu_B (w,v'')=p_3 p_2 p_1 p_0$, where
        \begin{align*}
        p_0=&\prod_{i\in G_w}b_i^2 ,& p_1=\prod_{i\in A_1}{\nu}_i(w_i,v''_i),\\
        p_2=&\prod_{i\in A_2}{\nu}_i(w_i,v'_i),& p_3=\prod_{i\in A_3}{\nu}_i(w_i,v'_i){\nu}_i(w_i,v''_i)
        \end{align*}
    \end{itemize}
\end{lem}

\begin{proof}
{[Note for referees: The proof can be omitted.] The proof of part (i) is straightforward. The proof of part (ii) is obtained using
\begin{align*}
    \nu_B (w,v')\nu_B (w,v'')=& \prod_{i=1}^{\ell}{\nu}_i(w_i,v'_i){\nu}_i(w_i,v''_i)\\
    =&  \prod_{j=0}^{3}\prod_{i\in A_j}{\nu}_i(w_i,v'_i){\nu}_i(w_i,v''_i),
\end{align*}
and the definition of $\nu_i$.
}
\end{proof}



\begin{pro}\label{mu.mu.sum} Let $v',v''\in V'_{\ell, \leq k;B}$. Then
\begin{itemize}
    \item [\rm (i)] ${\displaystyle\sum_{w\in V_{\ell k}}\nu_B(w,v')=
    \left\{
		                     \begin{array}{rl}
		                     		(-1)^{\ell-k} S_{\ell-k}(B),\,\, & \hbox{if $v'=g^{\ell}$,} \\
		                            	\\
		                            	0,\,\,  &\hbox{otherwise. }
		                      \end{array}
		                       \right.}$
    \item [\rm (ii)]${\displaystyle\sum_{w\in V_{\ell k}}\nu_B(w,v')\nu_B(w,v'')=
    \left\{
		                     \begin{array}{rl}
		   	{\displaystyle S_{\ell-k}(B(G_{v'}))\prod_{i\in G_{v'}} b_i \prod_{i\in \overline{G}_{v'}} (v'_i+{v'_i}^2)},\,\, & \hbox{if $v'=v''$,} \\
		                            	\\
		                            	0,\,\,  &\hbox{otherwise. }
		                      \end{array}
		                       \right.}$
\end{itemize}
\end{pro}
\begin{proof}{
\begin{itemize}
    \item [\rm (i)] For $w\in V_{\ell,k}$ we have
    $\nu_B(w,g^\ell)=\prod_{i\in G_w}(-b_i)$,
    hence we obtain
    $$\displaystyle\sum_{w\in V_{\ell k}}\nu_B(w,g^{\ell})=(-1)^{\ell-k} S_{\ell-k}(B)$$
which proves part(i) in the case $v'=g^{\ell}$.

Now suppose that $v' \in V_{\ell;\leq k;B}\setminus \{g^{\ell}\}$, hence for some $1\leq j \leq \ell$,\, $v'_j\neq g$. Then from
    \begin{align*}
    {\displaystyle \sum_{w\in V_{\ell k}}\nu_B(w,v')}=& \sum_{w_i=0}^{b_i-1} \prod_{i=1}^{\ell} \nu_i (w_i,v'_i)\\
   =&\prod_{i=1}^{\ell}\sum_{w_i=0}^{b_i-1}\nu_i (w_i,v'_i),
    \end{align*}
    by using $\sum_{w_j=0}^{b_j-1}\nu_j (w_j,v'_j)=0$,
     the right side is simplified to $0$, as required.

\item [\rm (ii)] To prove this part, observe that if $|A_0|<\ell-k$, each summand in the left, is zero and there is nothing to prove. So, let $|A_0|\geq \ell-k$, setting
$X_{\ell k}(B,G)=\{w\in V_{\ell k}(B): G_w=G\}$ we have
\begin{equation}\label{sumnunux}\sum_{w\in V_{\ell k}(B)}\nu_B(w,v') \nu_B(w,v'')=\sum_{G\in {{A_0}\choose{\ell-k}}} \sum_{w\in X_{\ell k}(B,G)}\nu_B(w,v') \nu_B(w,v'')\end{equation}

 First we compute the summand $\sum_{w\in X_{\ell k}(B,G)}\nu(w,v') \nu(w,v'')$, for a fixed $G\in {{A_0}\choose{\ell-k}}$.
 For this, without loss of generality, let $A_3=\{1, \ldots, a_3\}$, $A_2=\{a_3+1, \ldots, a_3+a_2\}$, $A_1=\{a_3+a_2+1, \ldots, a_3+a_2+a_1\}$ and
 $A_0=\{a_3+a_2+a_1+1, \ldots, \ell\}$, where $a_1,a_2,a_3$ are non-negative integers satisfying $a_1+a_2+a_3\leq k$ (this inequality is concluded from $|A_0|\geq \ell-k$). Moreover, without loss of generality, let $G=\{k+1, \ldots, \ell\}$. Now $w\in X_{\ell k}(B, G)$ can be factorized in the form $w=qrstg^{\ell-k}$, with $|q|=a_3$, $|r|=a_2$, $|s|=a_1$ and $|t|=a_0 - (\ell-k)$ and when $w$ runs over $X_{\ell k}(B,G)$, each of the words $q,r,s$ and $t$ runs over a proper set accordingly. By part (ii) of Lemma \ref{mu.mu} we obtain
 \begin{equation}
 \label{pis}
  \sum_{w\in X_{\ell k}(B,G)}\nu_B(w,v') \nu_B(w,v'')=P_3 P_2 P_1 P_0,
 \end{equation}
 where
         \begin{align*}
        P_0=&\prod_{i\in A_0}b_i \prod_{i\in G}b_i,& P_1=\prod_{i=a_3+a_2+1}^{a_3+a_2+a_1}\sum_{w_i=0}^{b_i-1}{\nu}_i(w_i,v'_i),\\
        P_2=&\prod_{i=a_3+1}^{a_3+a_2}\sum_{w_i=0}^{b_i-1}{\nu}_i(w_i,v''_i),&
        P_3=\prod_{i=1}^{a_3}\sum_{w_i=0}^{b_i-1}{\nu}_i(w_i,v'_i){\nu}_i(w_i,v''_i)
        \end{align*}

  [Note for referees: These three lines may be omitted.] Note that the above formula for $P_0$ is justified using Lemma \ref{mu.mu} (ii) as follows:
  {\small
  \begin{equation*}P_0=p_0\prod_{i=a_1+a_2+a_3+1}^k (\sum_{w_i=0}^{b_i-1}1)=(\prod_{i=k+1}^{\ell}b_i^2)\prod_{i=a_1+a_2+a_3+1}^k b_i=\prod_{i=k+1}^{\ell}b_i\prod_{i=a_1+a_2+a_3+1}^{\ell}b_i \end{equation*}
  }
Now, we consider two cases:
\\
{\bf Case 1.} $v'=v''$;
In this case $A_1=A_2=\emptyset$, hence $P_1=P_2=1$ and
$$ \sum_{w\in X_{\ell k}(B,G)}\nu_B^2(w,v')=P_3P_0$$
It is easily seen that in this case $P_3={\displaystyle \prod_{i\in {\overline G}_{v'}} (v'_i+{v'_i}^2)}$,
so,
\begin{align*}\sum_{w\in V_{\ell k}(B)}\nu_B(w,v') \nu_B(w,v'')=&\sum_{G\in {{A_0}\choose{\ell-k}}}  \sum_{w\in X_{\ell k}(B,G)}\nu_B(w,v') \nu_B(w,v'')\\
=&\sum_{G\in {{A_0}\choose{\ell-k}}} \prod_{i\in {\overline G}_{v'}} (v'_i+{v'_i}^2)\prod_{i\in A_0}b_i \prod_{i\in G}b_i\\
=&\prod_{i\in {\overline G}_{v'}} (v'_i+{v'_i}^2)\prod_{i\in A_0}b_i\sum_{G\in {{A_0}\choose{\ell-k}}}  \prod_{i\in G}b_i\\
=& \prod_{i\in {\overline G}_{v'}} (v'_i+{v'_i}^2)\prod_{i\in G_{v'}}b_i\,\,\,\,\, S_{\ell-k}(B(G_{v'}))
\end{align*}\\
{\bf Case 2.}  $v'\neq v''$;
It is easily proved that if $A_1\neq \emptyset$ then $P_1=0$ and if $A_2\neq \emptyset$ then $P_2=0$. Otherwise, if $A_1=A_2=\emptyset$, then
by Lemma \ref{A_is} (ii), $A_3\neq \emptyset$ and there exists $i\in A_3$ with $v'_i\neq v''_i$; It is easily proved that for this $i$,
$\sum_{w_i}\nu_i(w_i,v'_i) \nu_i(w_i,v''_i)=0$ thus $P_3=0$. Hence, the hypothesis $v'\neq v''$ implies that the right side of (\ref{pis}) is zero in either case, and we get the result by (\ref{sumnunux}).
\end{itemize}
}\end{proof}

\begin{pro}\label{someNuId}
\begin{itemize}
\item[\rm(i)] For any $u\in U_{\ell}$ and $v'\in V'_{\ell, \leq k}$ we have
\begin{equation}\label{nuId1}
				\sum_{y\in M_{\ell k}(u)} \nu_B(y,v')=(-1)^{\ell-k}S_{\ell-k}(B(G_{v'}))\, \nu_B(u,v'),
			\end{equation}
\item[\rm(ii)] For any $v\in V_{\ell k}$ and $v'\in V'_{\ell, \leq k}$ we have
\begin{equation}
\label{Az2}
\sum_{u\in M'_{\ell, k;B}(v)}\nu_B(u,v')=(-1)^{\ell-k} \nu_B(v,v') \end{equation}
\end{itemize}
\end{pro}
\begin{proof}
    {[Note for referees: The proof is similar to that of the Proposition 2 of \cite{kmer-b}, hence, it can be omitted.]
    \begin{itemize}
    \item[\rm(i)] If the summand $\nu(y,v')$ is nonzero, then $G_y \subseteq G_{v'}$, on the other hand the non-gaped positions of all such $y$'s are the same as $u$.
            Let $G_y= \{x_1, x_2, \ldots, x_{\ell-k}\}$, then
            $\nu_B(y,v')=(-1)^{\ell-k} b_{x_1}b_{x_1}\ldots b_{x_{\ell-k}}\, \nu_B(u,v')$. Therefore
            \begin{align*}\sum_{y\in M_{\ell k}(u)} \nu_B(y,v')=& (-1)^{\ell-k}\nu_B(u,v')\sum_{\{x_1, x_2, \ldots, x_{\ell-k}\}\subseteq G_{v'}}b_{x_1}b_{x_1}\ldots b_{x_{\ell-k}}\\
            =& (-1)^{\ell-k} S_{\ell-k}(B(G_{v'}))\, \nu_B(u,v'),
            \end{align*} as required.
    \item[\rm(ii)]
We distinguish two cases:

 			{\bf Case (a).} Suppose that $G_{v}\subseteq G_{v'}$. Then for any $u\in M'_{\ell, k;B}(v)$ we have	
				$$\nu_B(u,v')=\prod_{i\in \overline{G}_{v'}}\nu_i(u_i,v'_i)=\prod_{i\in \overline{G}_{v'}}\nu_i(v_i,v'_i)$$ and since there are totally  $\prod_{i\in G_{v'}} b_i $ such words $u$, the left side of equation (\ref{Az2}) equals
\begin{align*}
\prod_{i\in G_v} b_i \prod_{i\in \overline{G}_{v'}}\nu_i(v_i,v'_i)&=(-1)^{|G_v|}\prod_{i\in G_{v}}\nu_i(g,g)
\prod_{i \in G_{v'}\setminus G_v} \nu_i(v_i,g) \prod_{i\in \overline{G}_{v'}}\nu_i(v_i,v'_i)\\
&=(-1)^{\ell-k} \prod_{i=1}^{\ell}\nu_i(v_i,v'_i)\\
&=(-1)^{\ell-k} \nu_B(v,v'),
\end{align*}
as required.
			
			{\bf Case (b).} Suppose that $G_{v}\not \subseteq G_{v'}$, consequently $G_v \setminus G_{v'}\neq \emptyset$.
			Now for any $i\in G_v \setminus G_{v'}$ we have $\nu(v_i,v'_i)=0$, thus the right side of (\ref{Az2}) is $0$; The following argument shows that the left side is $0$ as well: The nonzero summands in the left side of (\ref{Az2}) are obtained from elements $u\in X$ where the subset $X\subseteq U_{\ell;B}$ is given by
			$$X=\{u\in U_{\ell;B} : u_i\leq v_i {\rm \,\, for \,\,} i\in G_{v}\setminus G_{v'} {\rm \,\,and\,\,} u_i=v_i {\rm \,\, for \,\,} i\in \overline{G}_v \}.$$
			Thus we obtain
			\begin{align*}
				\sum_{u\in M'_{\ell k;B}(v)} \nu_B(u,v')&=\sum_{u\in X} \nu_B(u,v')\\
				&=\sum_{u\in X} \prod_{i=1}^\ell  \nu_i(u_i,v'_i)\\
				&=\sum_{u\in X} \left( \prod_{i\in \overline{G}_v} \nu_i(u_i,v'_i) \prod_{i\in G_v \setminus G_{v'}} \nu_i(u_i,v'_i) \prod_{i\in G_v \cap G_{v'}} \nu_i(u_i,v'_i)\right)\\
				&=\sum_{u\in X} \left( \prod_{i\in \overline{G}_v} \nu_i(v_i,v'_i) \prod_{i\in G_v \setminus G_{v'}} \nu_i(u_i,v'_i) \prod_{i\in G_v \cap G_{v'}} \nu_i(u_i,g)\right)\\
				&=\left( \prod_{i\in \overline{G}_v} \nu_i(v_i,v'_i) \right) \left( \prod_{i\in G_v \setminus G_{v'}} \sum_{u_i=0}^{v'_i}\nu_i(u_i,v'_i)\right)
				\left( \prod_{i\in G_v \cap G_{v'}} \sum_{u_i=0}^{b_i-1}\nu_i(u_i,g)\right)\\
                &=\left( \prod_{i\in \overline{G}_v} \nu_i(v_i,v'_i) \right) \left( \prod_{i\in G_v \setminus G_{v'}} \sum_{u_i=0}^{v'_i}\nu_i(u_i,v'_i)\right)
				\left( \prod_{i\in G_v \cap G_{v'}} b_i\right)\\
				&=0
			\end{align*}
			The last identity holds because for any $i\in G_v\setminus G_{v'}$ we have $$\sum_{u_i=0}^{v'_i}\nu_i(u_i,v'_i)=\sum_{u_i=0}^{v'_i-1}1-v'_i=0.$$\\
Thus (\ref{Az2}) is true in either case.
    \end{itemize}
    }
    \end{proof}

\begin{definition} \label{defPQ} Let $u\in \Sigma_B=U_{\ell;B} $ and $v\in \Delta_B$. Then $P(u,v)$ and $Q(u,v)$ are defined as below
\begin{align*}
P(u,v)&=\{i:1\leq i\leq \ell , v_i\neq g, v_i=u_i\}\\
Q(u,v)&=\{i:1\leq i\leq \ell , v_i\neq g, v_i\neq u_i\}
\end{align*}
We denote $P(u,v)$ and $Q(u,v)$ by $P$ and $Q$, respectively, if there is no danger of confusion.
\end{definition}
With the above definition, it is obvious that
$$|P(u,v)|+|Q(u,v)|=|\overline{G}_v|.$$
Particularly, if $v\in V_{\ell,k}$ and $|P|=p$ then $|Q|=k-p$.

\begin{pro} \label{NuIdentity} Let $u\in \Sigma_B=U_{\ell;B} $ and $v\in \Delta_B$. Recall the notation of Definition \ref{defPQ}.
        \begin{itemize}
            \item[\rm(i)]  For $1\leq i\leq \ell $, let
                $$\phi_i(u,v)=\sum_{j=\max\{1,v_i\}}^{b_i-1} \frac{\nu_i(v_i,j)\nu_i(u_i,j)}{j(j+1)}.$$
            Then we have
                $$\phi_i(u,v)=\left\{
                             \begin{array}{ll}
                                    \frac{b_i-1}{b_i}, & \hbox{if $i\in P$;} \\
                                        \\
                                        \frac{-1}{b_i}, & \hbox{otherwise, i.e. if } i\in Q.
                              \end{array}
                               \right.$$
            \item[\rm(ii)] Let $G$ be a given subset of $\{1,\cdots,\ell\}$ with $G_v \subseteq G$. Then the following identity holds
            \begin{equation}
                \label{fracNuId} {\displaystyle \sum_{v'\in \Gamma_B\,\, ,G_{v'}=G} \frac{\nu_B(v,v')\nu_B(u,v')}{ {\displaystyle \prod_{i\in \overline{G}_{v'}} (v'_i+{v'_i}^2)}}
                }= \frac{\displaystyle (-1)^{|Q\setminus G|+|G_v|} \prod_{i \in G_v} b_i \prod_{i\in P \setminus G}(b_i-1)}{\displaystyle \prod_{i \in \overline{G}}b_i}.
            \end{equation}
        \end{itemize}
    \end{pro}

\begin{proof}{[Note for referees: The proof is similar to that of the Proposition 3 of \cite{kmer-b}, hence, it can be omitted.]
        \begin{itemize}
            \item[\rm(i)]  If $i\in P$ and $v_i>0$ then
            \begin{align*}
                \phi_i(u,v)&=\sum_{j=v_i}^{b_i-1} \frac{\nu_i(v_i,j)\nu_i(u_i,j)}{j(j+1)}\\
                &=\sum_{j=v_i}^{b_i-1} \frac{\nu_i(v_i,j)^2}{j(j+1)}\\
                &=\frac{v_i^2}{v_i(v_i+1)}+\sum_{j=v_i+1}^{b_i-1} \frac{1}{j(j+1)}\\
                &=\frac{v_i}{v_i+1}+\bigg(\frac{1}{v_i+1}-\frac{1}{b_i}\bigg)\\
                &=\frac{b_i-1}{b_i}.
            \end{align*}
            If $i\in P$ and $v_i=0$, then given $j\geq1$, we have
            $\nu_i(u_i,j)=\nu_i(v_i,j)=1$, hence,
            \begin{align*}
                \phi_i(u,v)&=\sum_{j=1}^{b_i-1} \frac{1}{j(j+1)}\\
                &=\frac{b_i-1}{b_i}.
            \end{align*}

             The case $i\in Q$ is done similarly.
            \item[\rm(ii)]
            Note that if $G_{v} \subseteq G_{v'}=G$ we have
            \begin{align*} \nu_B(v,v')\,\nu_B(u,v')&=\prod_{i=1}^{\ell} \nu_i(v_i,v'_i)\,\nu_i(u_i,v'_i)\\
            &=\prod_{i\in G_v} \nu_i(g,g)\nu_i(u_i,g) \prod_{i\in G\setminus G_v} \nu_i(v_i,g)\,\nu_i(u_i,g) \prod_{i\in \overline{G}} \nu_i(v_i,v'_i)\,\nu_i(u_i,v'_i)\\
            &=\prod_{i\in G_v} (-b_i) \prod_{i\in \overline{G}} \nu_i(v_i,v'_i)\,\nu_i(u_i,v'_i)\\
            &=(-1)^{|G_v|}\prod_{i\in G_v} b_i \prod_{i\in \overline{G}} \nu_i(v_i,v'_i)\,\nu_i(u_i,v'_i)            \end{align*}
            hence
            \begin{align*}
                {\displaystyle \sum_{v'\in \Gamma_B\,\, ,G_{v'}=G} \frac{\nu_B(v,v')\,\nu_B(u,v')}{ {\displaystyle \prod_{i\in \overline{G}_{v'}} (v'_i+{v'_i}^2)}}
                }
                &=(-1)^{|G_v|} {\displaystyle \prod_{i\in G_v} b_i} {\displaystyle \sum_{v'\in \Gamma_B\,\, ,G_{v'}=G} \,\,\prod_{i\in \overline{G}} \frac{ {\nu_i(v_i,v'_i)\nu_i(u_i,v'_i)}} { {v'_i(v'_i+1)}}
                }\\
                &=(-1)^{|G_v|} {\displaystyle \prod_{i\in G_v} b_i  \prod_{i \in \overline{G}} \,\,\sum_{j=\max\{1,v_i\}}^{b_i-1} \frac{\nu_i(v_i,j)\nu_i(u_i,j)}{j(j+1)}}\\
                &=(-1)^{|G_v|} {\displaystyle \prod_{i\in G_v} b_i  \prod_{i \in \overline{G}}\phi_i(u,v)}\\
                &=(-1)^{|G_v|} {\displaystyle \prod_{i\in G_v} b_i  \prod_{i \in \overline{G}\cap Q}\frac{-1}{b_i} \prod_{i \in \overline{G}\cap P}\frac{b_i-1}{b_i}}\\
                &=\frac{\displaystyle (-1)^{|Q \setminus G|+|G_v|} \prod_{i \in G_v} b_i \prod_{i\in P\setminus G}(b_i-1)}{\displaystyle \prod_{i \in \overline{G}}b_i}.
                \end{align*}

        \end{itemize}}
\end{proof}

\section{Orthonormal nonzero eigendecomposition of the matrices $AA^{\top}$ and $A^{\top}A$}\label{eigen sec}

In this section we give an orthonormal nonzero eigendecomposition of the matrices $AA^{\top}$ and $A^{\top}A$. Eigenvectors of these matrices are the elementary symmetric polynomials and the entries of the corresponding eigenvectors are given in terms of the function $\nu_B$. Using the properties of $\nu_B$ we show that these
eigenvectors are mutually orthogonal.

    \begin{definition} Let $n\leq k\leq \ell$ be integers. Given $B=(b_1,\cdots,b_{\ell})$ and $v'\in V'_{\ell,n;B}$, we define the column vector $x_{v'}^{\ell,k,n}$ as a vector whose rows are indexed by the elements of $V_{\ell,k;B}$ with entries $x_{v'}^{\ell,k,n}(w)=(-1)^{\ell-k}\nu_B(w,v')$. The column vector $z_{v'}^{\ell,n}$ is then defined as $z_{v'}^{\ell,n}=x_{v'}^{\ell,\ell,n}$; In other words, $z_{v'}^{\ell,n}$ is a column vector whose rows are indexed by elements $u$ of $U_{\ell;B}$ with entries $z_{v'}^{\ell,n}(u)=\nu_B(u,v')$. When there is no need to emphasize the parameters $\ell$, $k$ and $n$, we simply write $x_{v'}$ and $z_{v'}$.
    \end{definition}

    \begin{pro}\label{p-nrm-x}
Let $v' \in V'_{\ell, \leq k}$ then the following identity holds:
$$\parallel x_{v'}\parallel^2=  S_{\ell -k}(B(G_{v'}))\prod_{i\in {\overline{G}}(v')} (v'_i+{v'_i}^2)\prod_{i\in G_{v'}} b_i$$
\end{pro}
\begin{proof}{See proposition \ref{mu.mu.sum} (ii)}
\end{proof}

The following proposition contains a generalization of Proposition 2 of \cite{kmer-b}:\\

    \begin{pro} \label{matIdThm}
         Let $0\leq n\leq k\leq \ell$ and $v'\in V'_{\ell n}$. The following matrix identities hold.
           \begin{itemize}
            \item[\rm(i)] $A_{\ell k;B}^{\top}x_{v'}^{\ell k n}=S_{\ell-k}(B(G_{v'}))\, z_{v'}^{\ell n}$.
            \item[\rm(ii)] $A_{\ell k;B}z_{v'}^{\ell n}=x_{v'}^{\ell k n}$.
            \item[\rm(iii)] $A_{\ell k;B}A_{\ell k;B}^{\top} x_{v'}^{\ell k n}= S_{\ell-k}(B(G_{v'}))\, x_{v'}^{\ell k n}$.
            \item[\rm(iv)] $A_{\ell k;B}^{\top}A_{\ell k;B} z_{v'}^{\ell n}=S_{\ell-k}(B(G_{v'}))\, z_{v'}^{\ell n}$.
            \item[\rm(v)] For any two distinct words $v\in V'_{\ell n_1}$ and $u\in V'_{\ell n_2}$, the vectors
            $x_{v}^{\ell k n_1}$ and $x_{u}^{\ell k n_2}$ are orthogonal.
            \item[\rm(vi)] For any two distinct words $v\in V'_{\ell n_1}$ and $u\in V'_{\ell n_2}$, the vectors
            $z_{v}^{\ell n_1}$ and $z_{u}^{\ell n_2}$ are orthogonal.
        \end{itemize}
 \end{pro}
\begin{proof}
    {The proofs of (i) and (ii) are concluded from definitions of $x_{v'}$ and $z_{v'}$ and Proposition \ref{someNuId}. Combining (i) and (ii) yields (iii) and (iv). The proofs of (v) is concluded from Proposition \ref{mu.mu.sum}(ii). The same proposition yields part (vi) (by setting $k=\ell$).}
\end{proof}

\begin{pro}\label{egenvalueAAT,all}
Let $0\leq k\leq \ell$. Then the set $\{S_{\ell-k}(B(G_{v'})): v'\in V'_{\ell, \leq k}\}$ are all non-zero
eigenvalues of the matrix $A_{\ell k;B}A_{\ell k;B}^{\top}$. The set
$\{x_{v'}: v'\in V'_{\ell,\leq k}\}$ is a complete set of eigenvectors corresponding to nonzero eigenvalues.
Moreover, these eigenvectors are pairwise orthogonal.
\end{pro}

\begin{proof}
{First by Proposition \ref{matIdThm}, we conclude that for every $v'\in V'_{\ell, \leq k}$ we have
$A_{\ell k}A_{\ell k}^{\top} x_{v'}= S_{\ell-k}(B(G_{v'})) x_{v'}$. Hence $S_{\ell-k}(B(G_{v'}))$
is an eigenvalue of $A_{\ell k}A_{\ell k}^{\top}$ and $x_{v'}$ is its corresponding eigenvector. By Proposition \ref{matIdThm} (v), we know that these eigenvectors are pairwise orthogonal. To complete the proof,
We should prove that there are no more nonzero eigenvalues and there are no more independent eigenvectors corresponding to these eigenvalues.

Since the matrix $A_{\ell k;B}A_{\ell k;B}^{\top}$ is symmetric, positive semi-definite, all whose eigenvalues are nonnegative. So, to prove the Theorem it suffices to prove that
$$trace (A_{\ell k;B}A_{\ell k;B}^{\top})={\displaystyle\sum_{v'\in V'_{\ell n}}  }S_{\ell-k}(B(G_{v'})).$$
It can easily seen that $trace (A_{\ell k;B}A_{\ell k;B}^{\top})={{\ell}\choose k} {\displaystyle\prod_{i=1}^{\ell}b_i}$.
On the other hand, the summation of the nonzero eigenvalues of obtained from Proposition \ref{matIdThm} (iii), is simplified as
\begin{align*}
{\displaystyle\sum_{v'\in V'_{\ell n}}}  S_{\ell-k}(B(G_{v'}))&= \sum_{G\subseteq [{\ell}]: |G|\geq \ell-k} \,\,\,\
\sum_{v'\in V'_{\ell n}: G_{v'}=G}S_{\ell-k}(B(G_{v'}))\\
&=\sum_{{\overline{G}}\subseteq [\ell]:|{\overline{G}}|\leq k}\,\,\,\, \prod_{i\in {\overline{G}}}(b_i-1) S_{\ell-k}(B(G))\\
&={{\ell}\choose k} {\displaystyle\prod_{i=1}^{\ell}b_i}  ~~~\hbox{(by Lemma \ref{s_i x 2}),}~~~
\end{align*}
 which completes the proof.}
\end{proof}

\begin{remark} With the same assumptions of Proposition \ref{egenvalueAAT,all} and using a similar method, it is concluded from Proposition \ref{matIdThm} (iv) that a complete set of eigenvectors of $A_{\ell k;B}^{\top}A_{\ell k;B}$ corresponding to nonzero eigenvalues is $\{z_{v'}: v'\in V'_{\ell,\leq k}\}$. Furthermore, by \ref{matIdThm} (vi), these vectors are pairwise orthogonal.
\end{remark}

The following corollary is immediately concluded from the previous proposition.
 \begin{cor} {We have $\rk(A_{\ell,k;B}A_{\ell,k;B}^{\top})=R_k(B)$. Consequently, $\rk(A_{\ell,k;B})=R_k(B)$.}
\end{cor}

    \begin{definition} \label{Upsilon} Given $B=(b_1,\cdots,b_{\ell})$, we define $\Upsilon_{\ell, k;B}$ as a matrix whose rows and columns are indexed by the elements of $V_{\ell, k;B}$ and $V'_{\ell ,\leq k;B}$ respectively and whose entries are given by $\Upsilon_{\ell, k;B}(w,v')=(-1)^{|w|-|w|_g}\nu_B(w,v')$. In other words, the columns of $\Upsilon_{\ell, k;B}$ are exactly the vectors $x_{v'}$ for $v'\in V'_{\ell ,\leq k;B}$.
    \end{definition}

\begin{definition} \label{Q-E-D-def} We define the matrices $\Lambda$, $E$, $D$ and $Q_{\ell,k;B}$ as follows.
            \begin{align}
            \Lambda &= {\rm diag}(S_{\ell-k}(B(G_{v'})))_{v'\in V'_{\ell,\leq k;B}}\label{lambda}\\
                E&={\rm diag}(\frac{1}{||x_{v'}||})_{v'\in V'_{\ell,\leq k;B}},\label{edef}\\
                Q_{\ell k;B}&=\Upsilon_{\ell, k;B} E,\label{qdef}
            \end{align}
\end{definition}
By the above definition, we obtain  $Q_{\ell, k;B}={\rm diag}(\frac{x_{v'}}{||x_{v'}||})_{v'\in V'_{\ell,\leq k}}$ and using Proposition \ref{egenvalueAAT,all} we obtain the following orthonormal nonzero eigendecomposition for $A_{\ell k} A_{\ell k}^{\top}$.

    \begin{thm} \label{eigenDecomp}
         With the above definitions, the matrix $A_{\ell k;B} A_{\ell k;B}^{\top}$ admits the orthonormal nonzero eigendecomposition \begin{align}\label{qLamQt}
         A_{\ell k;B} A_{\ell k;B}^{\top}=Q_{\ell, k;B}\Lambda Q_{\ell, k;B}^{\top}.
         \end{align}
    \end{thm}

\section{Computing the entries of the matrices $W$ and $WA$}\label{w sec}
In this section we give a concrete description of the entries of matrices $W$ and $H$,
using the orthonormal nonzero eigendecomposition of the matrices $AA^{\top}$.

\begin{thm}\label{W-entries}
Let $u \in U_{\ell;B}$, $v \in V_{\ell,k;B}$. Moreover, with notation of Definition \ref{defPQ}, let $P=P(u,v)$ and $Q=Q(u,v)$. Then the entry $W_{\ell,k;B}(u,v)$, the Moore-Penrose pseudo-inverse of $A$, is given as below
\begin{equation}
\label{w-ent-f}
W_{\ell, k;B}(u,v)=\frac{1}{\displaystyle \prod_{i \in \overline{G}_v} b_i}\,\, \sum_{G,\, G_v \subseteq G \subseteq [\ell]}\frac{\displaystyle (-1)^{|Q\setminus G|}\prod_{i \in P\setminus G}(b_i-1)}{S_{\ell-k}(B(G))}
\end{equation}
\end{thm}

\begin{proof}
{Using Lemma \ref{W_nzed_A}, We have
\begin{align*}
W_{\ell,k;B} &= A_{\ell,k;B}^{\top}Q_{\ell,k;B}\Lambda^{-1}Q_{\ell,k;B}^{\top}\\
&= A_{\ell,k;B}^{\top} \Upsilon_{\ell, k;B} E \Lambda^{-1} E^{\top}
                \Upsilon_{\ell k;B}^{\top},\label{qLam-Qt}
                \end{align*}

Thus setting
\begin{align}
D & := E \Lambda^{-1}E^{\top} = E \Lambda^{-1}E,\\
C_{\ell, k;B} & := \Upsilon_{\ell, k;B} D \Upsilon_{\ell, k;B}^{\top},\label{Cdef}
\end{align}
we obtain
\begin{equation}\label{WAC}
W_{\ell,k;B}=A_{\ell,k;B}^{\top} C_{\ell, k;B}
\end{equation}

By definition of the matrices $D$, $E$ and $\Lambda$, we have $D= {\rm diag}(d_{v'})$, where
 $d_{v'}=\frac{1}{||x_{v'}||^2 \lambda_{v'}}$. Thus by (\ref{p-nrm-x}),

  \begin{align}
         d_{v'}&= \frac{1}{S_{\ell-k}^2(B(G_{v'})) {\displaystyle \prod_{i\in \overline{G}_{v'}} (v'_i+{v'_i}^2)\prod_{i\in G_{v'}}b_i}} \label{d_v'}
         \end{align}

By (\ref{Cdef}) we provide
\begin{equation}\label{Cent}
	C_{\ell, k;B}(y,v)=\sum_{v'\in V'_{\ell ,\leq k}} \nu_B(y,v')\nu_B(v,v')d_{v'}
\end{equation}
 Thus, we have
\begin{align*}
	W_{\ell, k;B}(u,v)&=\sum_{y\in M_{\ell k}(u)} C_{\ell, k;B}(y,v)~~~\hbox{(by  (\ref{WAC}))}\\
&=\sum_{y\in M_{\ell k}(u)}\sum_{v'\in V'_{\ell ,\leq k}} \nu_B(v,v')\, \nu_B(y,v') d_{v'}~~~\hbox{(by  (\ref{Cent}))}\\
&=\sum_{v'\in V'_{\ell ,\leq k}}\nu_B(v,v')\, d_{v'} \sum_{y\in M_{\ell k}(u)}\nu_B(y,v')\\
&=\sum_{v'\in V'_{\ell ,\leq k}}(-1)^{\ell-k}\nu_B(v,v')\, d_{v'}\, S_{\ell-k}(B(G_{v'}))\,\nu_B(u,v')~~~\hbox{(by  (\ref{nuId1}))}\\
&=(-1)^{\ell-k}\sum_{v'\in V'_{\ell ,\leq k}} \frac{\nu_B(v,v')\, \nu_B(u,v')}{S_{\ell-k}(B(G_{v'})) {\displaystyle \prod_{i\in \overline{G}_{v'}} (v'_i+{v'_i}^2)\prod_{i\in G_{v'}}b_i}}~~~\hbox{(by (\ref{d_v'}))}~~~\\
&=(-1)^{\ell-k}\sum_{G, G_v \subseteq G}\,\,\, \sum_{v'\in V'_{\ell ,\leq k},G_{v'}=G} \frac{\nu_B(v,v')\, \nu_B(u,v')}{S_{\ell-k}(B(G_{v'})) {\displaystyle \prod_{i\in \overline{G}_{v'}} (v'_i+{v'_i}^2)\prod_{i\in G_{v'}}b_i}}\\
&=(-1)^{\ell-k}\sum_{G,\, G_v \subseteq G \subseteq [\ell]}\,\frac{1}{S_{\ell-k}(B(G)) {\displaystyle  \prod_{i\in G}b_i} }  \,\, \sum_{v'\in \Gamma_B\, ,G_{v'}=G} \frac{\nu_B(v,v')\,\nu_B(u,v')}{ {\displaystyle \prod_{i\in \overline{G}_{v'}} (v'_i+{v'_i}^2)}}\\
&=(-1)^{\ell-k}\sum_{G,\, G_v \subseteq G \subseteq [\ell]}\,\frac{1}{S_{\ell-k}(B(G)) {\displaystyle \prod_{i\in G}b_i}}. \frac{\displaystyle (-1)^{\ell-k+|Q\setminus G|} \prod_{i \in G_v} b_i \prod_{i\in P\setminus G}(b_i-1)}{\displaystyle \prod_{i \in \overline{G}}b_i}  ~~~\hbox{(by (\ref{fracNuId}))}~~~\\
&=\frac{1}{\displaystyle \prod_{i \in \overline{G}_v} b_i}\,\, \sum_{G,\, G_v \subseteq G \subseteq [\ell]}\frac{\displaystyle (-1)^{|Q\setminus G|}\prod_{i \in P\setminus G}(b_i-1)}{S_{\ell-k}(B(G))}
\end{align*}
}
\end{proof}

\begin{thm}\label{G-entries} The matrix $H_{\ell,k;B}:=W_{\ell,k;B}A_{\ell,k;B}$ is a symmetric idempotent matrix, the sum of entries of any of whose rows (columns) equals $1$. Furthermore, for any $u,w \in U_{\ell;B}$, the entry $H_{\ell,k;B}(u,w)$ is given as below, where by using Definition \ref{defPQ}, $P=P(u,w)$ and $Q=Q(u,w)$.
\begin{equation}
\label{g-ent-f}
H_{\ell, k;B}(u,w)=\frac{1}{\displaystyle \prod_{i=1}^{\ell} b_i}\,\, \sum_{G \subseteq [\ell],\,\,\ell-k\leq |G|}{\displaystyle (-1)^{|Q\setminus G|}\prod_{i \in P\setminus G}(b_i-1)}
\end{equation}

\end{thm}
\begin{proof}
{It is easily seen that $H$ is a symmetric idempotent matrix (This is in fact implicit in the proof of Lemma \ref{W_nzed_A} . Before proving that the sum of entries of each row (columns) of $H$ equals $1$, we need some notations. Let $j=[1,1,\ldots, 1]^{\top}$ and $e_1=[1,0,0,\ldots,0]^{\top}$ be vectors of appropriate sizes. We let $x^{\ell k 0}= x_{g g \ldots g}^{\ell k 0}$ and $z^{\ell k 0}=z_{gg\ldots g}^{\ell k 0}$. As define earlier the columns of $Q$ are vectors $\frac{x_{v'}^{\ell k n}}{\|x_{v'}^{\ell k n}\|}$,
without loss of generality suppose that the vector $\frac{x^{\ell k 0}}{\|x^{\ell k 0}\|}$ be the first column of $Q$ and we denote the corresponding eigenvalue ($S_{\ell-k}(B)$) by $\lambda_1$.

We want to prove that
\begin{equation}\label{hj=j}
A^{\top} Q \Lambda^{-1} Q^{\top}A j=j
\end{equation}
first note that $z^{\ell k 0}=j$, hence $Aj= Az^{\ell k 0}=x^{\ell k 0}$ (by Proposition \ref{matIdThm}, part (i)).
Since, $\frac{x^{\ell k 0}}{\|x^{\ell k 0}\|}$ is the first row of $Q^{\top}$ and since, the rows of $Q^{\top}$ are pairwise orthogonal, we have $Q^{\top} x^{\ell k 0}=\|x^{\ell k 0}\|e_1$. So, we have
\begin{align*}
A^{\top} Q \Lambda^{-1} Q^{\top}A j &= \|x^{\ell k 0}\| A^{\top} Q \Lambda^{-1} e_1\\
&= \frac{\|x^{\ell k 0}\|}{\lambda_1} A^{\top} Q e_1 \,\,\,\,
\\
&=  \frac{\|x^{\ell k 0}\|}{\lambda_1} A^{\top} \frac{x^{\ell k 0}}{\|x^{\ell k 0}\|}\\
&= \frac{1}{\lambda_1} A^{\top} x^{\ell k 0}\\
&= z^{\ell k 0} {\rm\,\,\,(by\,\, Proposition\, \ref{matIdThm}\, (i))}\\
&= j
\end{align*}

Firstly note that if $v\sim w$ then for any subset $G_v\subseteq G \subseteq [\ell]$ we have
\begin{align*}
P(u,v)\setminus G&=P(u,w)\setminus G\\
Q(u,v)\setminus G&=Q(u,w)\setminus G
\end{align*}
Secondly, by $H_{\ell,k;B}=W_{\ell,k;B}A_{\ell,k;B}$ we have
\begin{equation*}
H(u,w)_{\ell,k;B}=\sum_{v\in V_{\ell,k;B},\,\,v\sim w} W_{\ell,k;B}(u,v)
\end{equation*}

Thirdly, the result of Theorem \ref{W-entries}, is rewritten as
\begin{equation*}
W_{\ell,k:B}(u,v)=\frac{1}{\displaystyle \prod_{i=1}^{\ell} b_i}\,\, \sum_{G,\, G_v \subseteq G \subseteq [\ell]}\frac{\displaystyle (-1)^{|Q(u,v)\setminus G|}{\displaystyle \prod_{i \in G_v} b_i}\prod_{i \in P(u,v)\setminus G}(b_i-1)}{S_{\ell-k}(B(G))}
\end{equation*}
Thus by the two last formulas we obtain
\begin{align*}
H_{\ell,k;B}(u,w)&=\frac{1}{\displaystyle \prod_{i=1}^{\ell}b_i}\sum_{v\in V_{\ell,k;B},\,\,v\sim w}\,\,\,\, \sum_{G,\, G_v \subseteq G \subseteq [\ell]}\frac{\displaystyle (-1)^{|Q(u,v)\setminus G|}{\displaystyle \prod_{i \in G_v} b_i}\prod_{i \in P(u,v)\setminus G}(b_i-1)}{S_{\ell-k}(B(G))}
\\
&={\displaystyle \frac{1}{\prod_{i=1}^{\ell}b_i}}\,\sum_{G\subseteq [\ell],\,\ell-k\leq |G|}\,\,\,\, \sum_{v\in V_{\ell,k;B},\,v\sim w,G_v\subseteq G}\frac{\displaystyle (-1)^{|Q(u,v)\setminus G|}{\displaystyle \prod_{i \in G_v} b_i}\prod_{i \in P(u,v)\setminus G}(b_i-1)}{S_{\ell-k}(B(G))}\\
&={\displaystyle \frac{1}{\prod_{i=1}^{\ell}b_i}}\,\sum_{G\subseteq [\ell],\,\ell-k\leq |G|}\,\,\,\, \sum_{v\in V_{\ell,k;B},\,v\sim w,G_v\subseteq G}\frac{\displaystyle (-1)^{|Q(u,w)\setminus G|}{\displaystyle \prod_{i \in G_v} b_i}\prod_{i \in P(u,w)\setminus G}(b_i-1)}{S_{\ell-k}(B(G))}\\
&={\displaystyle \frac{1}{\prod_{i=1}^{\ell}b_i}}\,\sum_{G\subseteq [\ell],\,\ell-k\leq |G|}\frac{\displaystyle (-1)^{|Q(u,w)\setminus G|}\prod_{i \in P(u,w)\setminus G}(b_i-1)}{S_{\ell-k}(B(G))} \sum_{v\in V_{\ell,k;B},\,v\sim w,G_v\subseteq G}{\displaystyle \,\,\,\, \prod_{i \in G_v} b_i}\\
\end{align*}
Now, considering the fact that the inner summation is exactly $S_{\ell-k}(B(G))$, we obtain (\ref{g-ent-f}).
}
\end{proof}

\begin{example} Let $B=(b_1,b_2)=(3,2)$, $\ell=2, k=1$. Then we have \\
$$A_{2,1;B}=\bordermatrix{&{\texttt{\color{blue}00}}&{\texttt{\color{blue}01}} &{\texttt{\color{blue}10}} &{\texttt{\color{blue}11}} &{\texttt{\color{blue}20}} &{\texttt{\color{blue}21}} \cr
                           {\texttt{\color{blue}0g}} &1 &1 &0 &0 &0 &0 \cr
                           {\texttt{\color{blue}1g}} &0 &0 &1 &1 &0 &0\cr
                           {\texttt{\color{blue}2g}} &0 &0 &0 &0 &1 &1\cr
                           {\texttt{\color{blue}g0}} &1 &0 &1 &0 &1 &0\cr
                           {\texttt{\color{blue}g1}} &0 &1 &0 &1 &0 &1}$$

{\rm Case $n=0$}
$$                           x_{gg}=\bordermatrix{&{\texttt{\color{blue}gg}} \cr
                           {\texttt{\color{blue}0g}} &2 \cr
                           {\texttt{\color{blue}1g}} &2 \cr
                           {\texttt{\color{blue}2g}} &2 \cr
                           {\texttt{\color{blue}g0}} &3 \cr
                           {\texttt{\color{blue}g1}} &3 \cr} \qquad \quad
                             z_{gg}=\bordermatrix{&{\texttt{\color{blue}gg}} \cr
                           {\texttt{\color{blue}00}} &1 \cr
                           {\texttt{\color{blue}01}} &1 \cr
                           {\texttt{\color{blue}10}} &1 \cr
                           {\texttt{\color{blue}11}} &1 \cr
                           {\texttt{\color{blue}20}} &1 \cr
                           {\texttt{\color{blue}21}} &1 \cr} \qquad \quad
$$

$$                          AA^{\top}x_{gg}=5x_{gg} , \,\,\,\,\,\,
A^{\top}A z_{gg}=5 z_{gg}. $$

{\rm Case $n=1$:} in this case $V'=\{1g, 2g, g1\}$ and
$$                           x_{1g}=\bordermatrix{&{\texttt{\color{blue}1g}} \cr
                           {\texttt{\color{blue}0g}} &2 \cr
                           {\texttt{\color{blue}1g}} &-2 \cr
                           {\texttt{\color{blue}2g}} &0 \cr
                           {\texttt{\color{blue}g0}} &0 \cr
                           {\texttt{\color{blue}g1}} &0 \cr} \qquad \quad
                            x_{2g}=\bordermatrix{&{\texttt{\color{blue}2g}} \cr
                           {\texttt{\color{blue}0g}} &2 \cr
                           {\texttt{\color{blue}1g}} &2 \cr
                           {\texttt{\color{blue}2g}} &-4 \cr
                           {\texttt{\color{blue}g0}} &0 \cr
                           {\texttt{\color{blue}g1}} &0 \cr} \qquad \quad
                            x_{g1}=\bordermatrix{&{\texttt{\color{blue}g1}} \cr
                           {\texttt{\color{blue}0g}} &0 \cr
                           {\texttt{\color{blue}1g}} &0 \cr
                           {\texttt{\color{blue}2g}} &0 \cr
                           {\texttt{\color{blue}g0}} &3 \cr
                           {\texttt{\color{blue}g1}} &-3 \cr} \qquad \quad.
$$

$$AA^{\top}x_{1g}=2x_{1g} ,\,\,\,
AA^{\top}x_{2g}=2x_{2g} ,\,\,\,
AA^{\top}x_{g1}=3x_{g1}. $$

$$                          z_{1g}=\bordermatrix{&{\texttt{\color{blue}1g}} \cr
                           {\texttt{\color{blue}00}} &1 \cr
                           {\texttt{\color{blue}01}} &1 \cr
                           {\texttt{\color{blue}10}} &-1 \cr
                           {\texttt{\color{blue}11}} &-1 \cr
                           {\texttt{\color{blue}20}} &0 \cr
                           {\texttt{\color{blue}21}} &0 \cr} \qquad \quad
                            z_{2g}=\bordermatrix{&{\texttt{\color{blue}2g}} \cr
                           {\texttt{\color{blue}00}} &1 \cr
                           {\texttt{\color{blue}01}} &1 \cr
                           {\texttt{\color{blue}10}} &1 \cr
                           {\texttt{\color{blue}11}} &1 \cr
                           {\texttt{\color{blue}20}} &-2 \cr
                           {\texttt{\color{blue}21}} &-2 \cr} \qquad \quad
                            z_{g1}=\bordermatrix{&{\texttt{\color{blue}g1}} \cr
                           {\texttt{\color{blue}00}} &1 \cr
                           {\texttt{\color{blue}01}} &-1 \cr
                           {\texttt{\color{blue}10}} &1 \cr
                           {\texttt{\color{blue}11}} &-1 \cr
                           {\texttt{\color{blue}20}} &1 \cr
                           {\texttt{\color{blue}21}} &-1 \cr} \qquad \quad$$

$$A^{\top}Az_{1g}=2z_{1g} ,\,\,\,
A^{\top}Az_{2g}=2z_{2g} ,\,\,\,
A^{\top}Az_{g1}=3z_{g1}. $$

Matrices $\Upsilon, \Lambda, E, W$ and $H$ are as follows
$$\Upsilon_{2,1;B}=\bordermatrix{&{\texttt{\color{blue}gg}}&{\texttt{\color{blue}1g}} &{\texttt{\color{blue}2g}} &{\texttt{\color{blue}g1}} \cr
                           {\texttt{\color{blue}0g}} &2 &2 &2 &0 \cr
                           {\texttt{\color{blue}1g}} &2 &-2 &2 &0 \cr
                           {\texttt{\color{blue}2g}} &2 &0 &-4 &0 \cr
                           {\texttt{\color{blue}g0}} &3 &0 &0 &3 \cr
                           {\texttt{\color{blue}g1}} &3 &0 &0 &-3 },$$

$\Lambda=diag(5,2,2,3)$ and $E= diag(\frac{1}{\sqrt{30}},\frac{1}{\sqrt{8}},\frac{1}{\sqrt{24}},\frac{1}{\sqrt{18}})$.

$$Q_{2,1;B}=\bordermatrix{&{\texttt{\color{blue}gg}}&{\texttt{\color{blue}1g}} &{\texttt{\color{blue}2g}} &{\texttt{\color{blue}g1}} \cr
                           {\texttt{\color{blue}0g}} &\frac{2}{\sqrt{30}} &\frac{2}{\sqrt{8}} &\frac{2}{\sqrt{24}} &0 \cr
                           {\texttt{\color{blue}1g}} &\frac{2}{\sqrt{30}} &\frac{-2}{\sqrt{8}} &\frac{2}{\sqrt{24}} &0 \cr
                           {\texttt{\color{blue}2g}} &\frac{2}{\sqrt{30}} &0 &\frac{-4}{\sqrt{24}} &0 \cr
                           {\texttt{\color{blue}g0}} &\frac{3}{\sqrt{30}} &0 &0 &\frac{3}{\sqrt{18}} \cr
                           {\texttt{\color{blue}g1}} &\frac{3}{\sqrt{30}} &0 &0 &\frac{-3}{\sqrt{18}} }$$

$$W_{2,1;B}=\frac{1}{30}\bordermatrix{&{\texttt{\color{blue}0g}}&{\texttt{\color{blue}1g}} &{\texttt{\color{blue}2g}} &{\texttt{\color{blue}g0}} &{\texttt{\color{blue}g1}}\cr
                           {\texttt{\color{blue}00}} &12 &-3 &-3 &8  &-2 \cr
                           {\texttt{\color{blue}01}} &12 &-3 &-3 &-2 &8  \cr
                           {\texttt{\color{blue}10}} &-3 &12 &-3 &8  &-2 \cr
                           {\texttt{\color{blue}11}} &-3 &12 &-3 &-2 &8  \cr
                           {\texttt{\color{blue}20}} &-3 &-3 &12 &8  &-2 \cr
                           {\texttt{\color{blue}21}} &-3 &-3 &12 &-2 &8 },\,\,\,\,
                           H_{2,1;B}=\frac{1}{6}\bordermatrix{&{\texttt{\color{blue}00}}&{\texttt{\color{blue}01}} &{\texttt{\color{blue}10}} &{\texttt{\color{blue}11}} &{\texttt{\color{blue}20}} &{\texttt{\color{blue}21}}\cr
                           {\texttt{\color{blue}00}} &4  &2  &1  &-1 &1  &-1 \cr
                           {\texttt{\color{blue}01}} &2  &4  &-1 &1  &-1 &1  \cr
                           {\texttt{\color{blue}10}} &1  &-1 &4  &2  &1  &-1 \cr
                           {\texttt{\color{blue}11}} &-1 &1  &2  &4  &-1 &1  \cr
                           {\texttt{\color{blue}20}} &1  &-1 &1  &-1 &4  &2  \cr
                           {\texttt{\color{blue}21}} &-1 &1  &-1 &1  &2  &4 }$$

\end{example}

{\bf Acknowledgement.}


\begin{thebibliography}{99}

\bibitem{CameronNotes}  P. J. Cameron, (2003) Notes on Counting. http://www.maths.qmul.ac.uk/~pjc/notes\\ /counting.pdf. Accessed 25 January 2012.

\bibitem{cox} D. A Cox, J. Little, D. O'Shea, {\rm Ideals, varieties, and algorithms: an introduction to computational algebraic geomerty
and commutative algebra, (3rd edition)}, Springer, 2007.

\bibitem{Delsarte} P. Delsarte, {\it Beyond the orthogonal array concept},
European J. Combin. (2004). {\bf 25}, 187-198

\bibitem{Delsarte association} P. Delsarte, {\it Association schemes and $t$-designs in regular semilattices},
J. Combin. Theory Ser. A 20 (1976)

\bibitem{handbooklin} L. Han and M. Neumann, {\it Inner Product Spaces, Orthogonal Projection, Least Squares
and Singular Value Decomposition} In: Leslie Hogben (eds) {\it Handbook of Linear Algebra, 2nd Edition (Discrete Mathematics and Its Applications)} Boca Raton, FL: CRC Press; (2013).

\bibitem{Enhanced kmer-b}
M. Ghandi, D. Lee, M. Mohammad-Noori, and M. A. Beer,
{\it Enhanced regulatory sequence prediction using gapped k-mer features},
PLoS Comput Biol, 2014. 10(7): p. e1003711

\bibitem{kmer-b}
M. Ghandi, M. Mohammad-Noori, and M. A. Beer,
 {\it Robust $k$-mer frequency estimation using gapped $k$-mers},
 Journal of mathematical biology. {\bf 60}, Issue 2 (2014), 469-500.


\bibitem{knuth} R. L. Graham, D. E. Knuth, and O. Patashnik, {\rm Concrete Mathematics: A Foundation for Computer Science (Second Edition)}, Addison Wesley Publishing Company, 1994.	
	
\bibitem{linalg} M. Marcus and H. Minc, Introduction to Linear Algebra. New York: Dover, p. 182, 1988.

\bibitem{linalg PDM} M. Marcus and H. Minc, "Positive Definite Matrices." §4.12 in A Survey of Matrix Theory and Matrix Inequalities. New York: Dover, p. 69, 1992.

\bibitem{Terwilliger} P. Terwilliger, {\it The incidence algebra of a uniform poset} In:
Codeng Theory and Design Theory Part I: Coding Theory, IMA Volumes in Mathematics and its Applications, vol. (20),
Springer, New York, 1990, 193-212.

\bibitem{website} Website

\end{thebibliography}
\end{document}